\documentclass[12pt, a4paper]{amsart}  

\newif\ifpersonal
\newif\ifarxiv
\RequirePackage[l2tabu, orthodox]{nag} 
\allowdisplaybreaks[1]

\usepackage{tikz}
\tikzset{math3d/.style={x= {(1cm,0cm)}, z={(0cm,1cm)},y={(0.353cm,0.353cm)}}}
\usetikzlibrary{decorations.pathreplacing}

\usepackage{tikz,tikz-cd}
\usetikzlibrary{arrows,matrix}
\usetikzlibrary{intersections}  
\usetikzlibrary{calc}  

\newcommand{\FourCurveVexNorm}[5]{%

  \pgfmathsetmacro{\cst}{1/3}
  \pgfmathsetmacro{\Px}{#1}
  \pgfmathsetmacro{\Py}{#2}
  \pgfmathsetmacro{\dx}{#3}
  \pgfmathsetmacro{\dy}{#4}
  \pgfmathsetmacro{\Pangle}{atan(\dy/(\cst*\dx))-45}
  \pgfmathsetmacro{\Qangle}{atan(\cst*\dy/\dx)-120}
  
  \coordinate (P) at (\Px,\Py);
  \coordinate (Q) at (\Px+\dx,\Py+\dy);
  
  \draw[line width=1pt,color=black,
  name path=#5] (P) to[out=\Pangle,in=\Qangle] (Q);

  \foreach \i in {1,2,3,4}{%
    \pgfmathsetmacro{\My}{\Py+\i*\dy/6}
    \path[name path=tmpLine] (\Px,\My) -- ++(\dx,0);
    \draw [thick,fill=black,
    name intersections={of=#5 and tmpLine, by={M}}]
    (M) circle (1pt) node[left] {};
  }
}

\newcommand{\FourCurveVex}[5]{%

  \pgfmathsetmacro{\cst}{1/3}
  \pgfmathsetmacro{\Px}{#1}
  \pgfmathsetmacro{\Py}{#2}
  \pgfmathsetmacro{\dx}{#3}
  \pgfmathsetmacro{\dy}{#4}
  \pgfmathsetmacro{\Pangle}{atan(\dy/(\cst*\dx))-45}
  \pgfmathsetmacro{\Qangle}{atan(\cst*\dy/\dx)-120}
  
  \coordinate (P) at (\Px,\Py);
  \coordinate (Q) at (\Px+\dx,\Py+\dy);
  
  \draw[line width=1pt,color=black,
  name path=#5] (P) to[out=\Pangle,in=\Qangle] (Q);

  \foreach \i in {1,2,3,4}{%
    \pgfmathsetmacro{\My}{\Py+\i*\dy/5}
    \path[name path=tmpLine] (\Px,\My) -- ++(1,0);
    \draw [thick,fill=black,
    name intersections={of=#5 and tmpLine, by={M}}]
    (M) circle (1pt) node[left] {};
  }
\path[name path=tmpLine] (\Px,\Py+4*\dy/5) -- ++(1,0);
    \draw [thick,    name intersections={of=#5 and tmpLine, by={M}}]
    (M) circle (3.5pt) node[left] {};
\path[name path=tmpLine] (\Px,\Py+1.5*\dy/5) -- ++(1,0);
    \draw [thick,    name intersections={of=#5 and tmpLine, by={M}}]
    (M) circle (3.5pt) node[left] {};
}

\newcommand{\FourCurveVexSmall}[5]{%

  \pgfmathsetmacro{\cst}{1/3}
  \pgfmathsetmacro{\Px}{#1}
  \pgfmathsetmacro{\Py}{#2}
  \pgfmathsetmacro{\dx}{#3}
  \pgfmathsetmacro{\dy}{#4}
  \pgfmathsetmacro{\Pangle}{atan(\dy/(\cst*\dx))-45}
  \pgfmathsetmacro{\Qangle}{atan(\cst*\dy/\dx)-120}
  
  \coordinate (P) at (\Px,\Py);
  \coordinate (Q) at (\Px+\dx,\Py+\dy);
  
  \draw[line width=1pt,color=black,
  name path=#5] (P) to[out=\Pangle,in=\Qangle] (Q);

  \foreach \i in {1,2,3,4}{%
    \pgfmathsetmacro{\My}{\Py+\i*\dy/5}
    \path[name path=tmpLine] (\Px,\My) -- ++(1,0);
    \draw [thick,fill=black,
    name intersections={of=#5 and tmpLine, by={M}}]
    (M) circle (1pt) node[left] {};
  }
\path[name path=tmpLine] (\Px,\Py+1.5*\dy/5) -- ++(1,0);
    \draw [thick,    name intersections={of=#5 and tmpLine, by={M}}]
    (M) circle (2.5pt) node[left] {};
}

\newcommand{\ThreeCurveVexNorm}[5]{%

  \pgfmathsetmacro{\cst}{1/3}
  \pgfmathsetmacro{\Px}{#1}
  \pgfmathsetmacro{\Py}{#2}
  \pgfmathsetmacro{\dx}{#3}
  \pgfmathsetmacro{\dy}{#4}
  \pgfmathsetmacro{\Pangle}{atan(\dy/(\cst*\dx))-45}
  \pgfmathsetmacro{\Qangle}{atan(\cst*\dy/\dx)-120}
  
  \coordinate (P) at (\Px,\Py);
  \coordinate (Q) at (\Px+\dx,\Py+\dy);
  
  \draw[line width=1pt,color=black,
  name path=#5] (P) to[out=\Pangle,in=\Qangle] (Q);

  \foreach \i in {1,2,3}{%
    \pgfmathsetmacro{\My}{\Py+\i*\dy/4}
    \path[name path=tmpLine] (\Px,\My) -- ++(\dx+1,0);
    
    \draw [thick,fill=black,
    name intersections={of=#5 and tmpLine, by={M}}]
    (M) circle (1.5pt) node[left] {};
  }
}

\newcommand{\ThreeCurveVex}[5]{%

  \pgfmathsetmacro{\cst}{1/3}
  \pgfmathsetmacro{\Px}{#1}
  \pgfmathsetmacro{\Py}{#2}
  \pgfmathsetmacro{\dx}{#3}
  \pgfmathsetmacro{\dy}{#4}
  \pgfmathsetmacro{\Pangle}{atan(\dy/(\cst*\dx))-45}
  \pgfmathsetmacro{\Qangle}{atan(\cst*\dy/\dx)-120}
  
  \coordinate (P) at (\Px,\Py);
  \coordinate (Q) at (\Px+\dx,\Py+\dy);
  
  \draw[line width=1pt,color=black,
  name path=#5] (P) to[out=\Pangle,in=\Qangle] (Q);

  \foreach \i in {1,2,3}{%
    \pgfmathsetmacro{\My}{\Py+\i*\dy/4}
    \path[name path=tmpLine] (\Px,\My) -- ++(\dx+1,0);
    
    \draw [thick,fill=black,
    name intersections={of=#5 and tmpLine, by={M}}]
    (M) circle (1.5pt) node[left] {};
  }
}

\newcommand{\ZeroCurveVex}[5]{%

  \pgfmathsetmacro{\cst}{1/3}
  \pgfmathsetmacro{\Px}{#1}
  \pgfmathsetmacro{\Py}{#2}
  \pgfmathsetmacro{\dx}{#3}
  \pgfmathsetmacro{\dy}{#4}
  \pgfmathsetmacro{\Pangle}{atan(\dy/(\cst*\dx))-45}
  \pgfmathsetmacro{\Qangle}{atan(\cst*\dy/\dx)-120}
  
  \coordinate (P) at (\Px,\Py);
  \coordinate (Q) at (\Px+\dx,\Py+\dy);
  
  \draw[line width=1pt,color=black,
  name path=#5] (P) to[out=\Pangle,in=\Qangle] (Q);

}

\newcommand{\TwoCurveCave}[5]{%

  \pgfmathsetmacro{\cst}{1/3}
  \pgfmathsetmacro{\Px}{#1}
  \pgfmathsetmacro{\Py}{#2}
  \pgfmathsetmacro{\dx}{#3}
  \pgfmathsetmacro{\dy}{#4}
  \pgfmathsetmacro{\Pangle}{atan(\dy/(\cst*\dx))+50}
  \pgfmathsetmacro{\Qangle}{atan(\cst*\dy/\dx)+130}
  
  \coordinate (P) at (\Px,\Py);
  \coordinate (Q) at (\Px+\dx,\Py+\dy);
  
  \draw[line width=1pt,color=black,
  name path=#5] (P) to[out=\Pangle,in=\Qangle] (Q);

  \foreach \i in {1,2}{%
    \pgfmathsetmacro{\My}{\Py+\i*\dy/3}
    \path[name path=tmpLine] (\Px,\My) -- ++(\dx,0);
    
    \draw [thick,fill=black,
    name intersections={of=#5 and tmpLine, by={M}}]
    (M) circle (1.5pt) node[left] {};
  }
}

\newcommand{\ThreeCurveCave}[5]{%

  \pgfmathsetmacro{\cst}{1/3}
  \pgfmathsetmacro{\Px}{#1}
  \pgfmathsetmacro{\Py}{#2}
  \pgfmathsetmacro{\dx}{#3}
  \pgfmathsetmacro{\dy}{#4}
  \pgfmathsetmacro{\Pangle}{atan(\dy/(\cst*\dx))+50}
  \pgfmathsetmacro{\Qangle}{atan(\cst*\dy/\dx)+130}
  
  \coordinate (P) at (\Px,\Py);
  \coordinate (Q) at (\Px+\dx,\Py+\dy);
  
  \draw[line width=1pt,color=black,
  name path=#5] (P) to[out=\Pangle,in=\Qangle] (Q);

  \foreach \i in {1,2,3}{%
    \pgfmathsetmacro{\My}{\Py+\i*\dy/4}
    \path[name path=tmpLine] (\Px,\My) -- ++(\dx,0);
    
    \draw [thick,fill=black,
    name intersections={of=#5 and tmpLine, by={M}}]
    (M) circle (1.5pt) node[left] {};
  }
}

\newcommand{\TwoCurveCaveVert}[5]{%

  \pgfmathsetmacro{\cst}{1/3}
  \pgfmathsetmacro{\Px}{#1}
  \pgfmathsetmacro{\Py}{#2}
  \pgfmathsetmacro{\dx}{#3}
  \pgfmathsetmacro{\dy}{#4}
  \pgfmathsetmacro{\Pangle}{atan(\dy/(\cst*\dx))+50}
  \pgfmathsetmacro{\Qangle}{atan(\cst*\dy/\dx)+130}
  
  \coordinate (P) at (\Px,\Py);
  \coordinate (Q) at (\Px+\dx,\Py+\dy);
  
  \draw[line width=1pt,color=black,
  name path=#5] (P) to[out=\Pangle,in=\Qangle] (Q);

  \foreach \i in {1,2}{%
    \pgfmathsetmacro{\Mx}{\Px+\i*\dx/3}
    \path[name path=tmpLine] (\Mx,\Py) -- (\Mx,\Py+1);
    
    \draw [thick,fill=black,
    name intersections={of=#5 and tmpLine, by={M}}]
    (M) circle (1.5pt) node[left] {};
  }
}

\ifpersonal
\newcommand*{\personal}[1]{\textcolor{blue}{(Personal: #1)}}
\newcommand*{\todo}[1]{\textcolor{red}{(Todo: #1)}}
\else
\newcommand*{\personal}[1]{}
\newcommand*{\todo}[1]{}
\fi

\usepackage{amsmath}
 \usepackage{microtype,fixltx2e,lmodern} 
\usepackage{enumerate,comment,braket,xspace,tikz-cd,csquotes} 
\usepackage{amsxtra}
\usepackage{amscd}
\usepackage{amsthm}
\usepackage{amsfonts}
\usepackage{amssymb}
\usepackage{comment}
\usepackage{eucal}
\usepackage[all]{xy}
\usepackage{graphicx,psfrag}
\usepackage{hyperref}
\usepackage{color}
\usepackage{enumitem}
\usepackage[english]{babel}
\usepackage[utf8]{inputenc}
\usepackage{mathtools}
\usepackage{epstopdf}
\usepackage{mdframed}

\newcommand{\cocor}{\mathbf{CoCorr}}

\newcommand{\dSt}{\mathbf{dSt}_{\mathbf{C}}}

\DeclareMathOperator{\End}{End}
\DeclareMathOperator{\NE}{NE}

\DeclareMathOperator{\cst}{cst}%
\DeclareMathOperator{\pt}{pt}%
\DeclareMathOperator{\Td}{Td}%
\DeclareMathOperator{\Hom}{Hom}%
\DeclareMathOperator{\Id}{id}%
\DeclareMathOperator{\Spec}{Spec}
\DeclareMathOperator{\Aut}{Aut}

\DeclareMathOperator{\vir}{vir}%

\DeclareMathOperator{\ch}{ch}%
\newtheorem{prop}[equation]{Proposition}

\newtheorem{cor}[equation]{Corollary}
\newtheorem{lem}[equation]{Lemma}
\newtheorem{defn}[equation]{Definition}

\newtheorem{thm}[equation]{Theorem}

\numberwithin{equation}{subsection}

\theoremstyle{remark}
\newtheorem{ex}[equation]{Example}
\newtheorem{remark}[equation]{Remark}
\newtheorem{example}[equation]{Example}

\newcommand{\Grps}{(\mathrm{Grps})}

\newcommand{\Aff}{(\mathrm{Aff-sch})}

\DeclareMathOperator{\pres}{pre}
\DeclareMathOperator{\fake}{fake}
\DeclareMathOperator{\Sym}{Sym}

\DeclareMathOperator{\DAG}{DAG}
\DeclareMathOperator{\POT}{POT}

\DeclareMathOperator{\ev}{ev}

\DeclareMathOperator{\colim}{\mathrm{colim}}
\DeclareMathOperator{\stab}{\mathrm{Stab}}

\DeclareMathOperator{\Map}{\mathrm{Map}}

\DeclareMathOperator{\Qcoh}{\mathrm{Qcoh}}
\DeclareMathOperator{\Perf}{\mathrm{Perf}}

\DeclareMathOperator{\corr}{cor}

\newcommand{\dst}{\mathrm{dst}_k}

\def\BarrBeckTriangle#1{\innerBarrBeckTriangle(#1)}
\def\innerBarrBeckTriangle(#1;#2;#3;#4;#5;#6) {$$\xymatrix{ #1  \ar[rr]^{#4}\ar[dr]_{#6}&&  \ar[dl]^{#5} #2  \\   &#3 &  }$$}

\begin{document}

\title{Gromov-Witten theory with derived algebraic geometry}
\author{Etienne Mann}

\address{Etienne Mann, Universit\'e d'Angers, Département de mathématiques
Bâtiment I
Faculté des Sciences
2 Boulevard Lavoisier
F-49045 Angers cedex 01
France }
\email{etienne.mann@univ-angers.fr }

\author{Marco Robalo}

\address{Marco Robalo, Sorbonne  Université. Université Pierre  et  Marie  Curie,
Institut Mathématiques de Jussieu Paris Rive Gauche, CNRS, Case 247, 4, place Jussieu,
75252 Paris Cedex 05, France }
\email{marco.robalo@imj-prg.fr}

  \thanks{E.M is supported by the grant of the Agence Nationale de la
    Recherche ``New symmetries on Gromov-Witten theories'' ANR- 09-JCJC-0104-01. and "SYmétrie
    miroir et SIngularités irrégulières provenant de la PHysique "ANR-13-IS01-0001-01/02 and project
  ``CatAG''ANR-17-CE40-0014}
    
  \thanks{M. R was supported by a Postdoctoral Fellowship of the Fondation Sciences Mathematiques de
    Paris and ANR ``CatAG''ANR-17-CE40-0014}

\begin{abstract}
In this survey we add two new results that are not in our paper \cite{2015arXiv150502964M}.
Using the idea of brane actions discovered by To\"en, we construct a lax associative action of the
operad of stable curves of genus zero on a smooth variety $X$ seen as an object in correspondences in derived stacks. This action encodes the Gromov-Witten theory of $X$ in purely geometrical terms.

\end{abstract}


\personal{
\begin{center}
LIST OF TODOs TEST
\end{center}
\vspace{1cm}
Solve  PROP \ref{prop-compatiblePerfCoh}\\
1) the fact these are maps of stacks and not derived schemes, which are not representable by derived schemes
2)Why proper quasi-smooth (why of finite presentation?) pullbacks preserve bounded coherent? Meaning, why are they of finite tor amplitude.\\
\vspace{1cm}
THM \ref{thmbranes}\\
1) Re-check the proof specially the verification it is a cocartesian fibration\\
\vspace{1cm}
H-descent: I think the proof is correct because we are only describing the lax structure and for this we only need the structure sheaf. As the maps  in the diagram are closed immersions, therefore proper, then they preserve bounded coherent.\\ 
1) 
\newpage
}

\maketitle
\tableofcontents


\section{Introduction}
\label{sec:introduction}


This paper is a survey\footnote{We add two new results Theorem \ref{thm:colim} and \ref{thm:orientation}} of \cite{2015arXiv150502964M}.  We explain without technical details the
ideas of \cite{2015arXiv150502964M} where we use derived algebraic geometry to redefine
Gromov-Witten invariants and highlight the hidden operad picture.

Gromov-Witten invariants were introduced by Kontsevich and Manin in algebraic geometry in \cite{KMgwqceg, MR1363062}. The foundations were then completed by Behrend,
Fantechi and Manin in \cite{Behrend-Manin-stack-stable-mapGWI-1996},
\cite{MR1437495} and \cite{MR1431140}. In symplectic
geometry, the definition is due to Y. Ruan and G. Tian in \cite{RTmqc}, \cite{MR1390655} and
\cite{MR1483992}.  Mathematicians developed several
techniques to compute them: via a localization
formula proved by Graber and Pandharipande in
\cite{MR1666787}, via a degeneration formula proved
by J. Li in \cite{MR1938113} and another one called quantum Lefschetz proved by Coates-Givental
\cite{Givental-Coates-2007-QRR} and Tseng \cite{tseng_orbifold_2010}.

 These invariants can be encoded using different
mathematical structures: quantum products, cohomological field theories (Kontsevich-Manin in
\cite{KMgwqceg}), Frobenius manifolds (Dubrovin in \cite{Dtft}), Lagrangian cones and Quantum
$D$-modules (Givental \cite{MR2115767}), variations of non-commutative Hodge structures (Iritani
\cite{Iritani-2009-Integral-structure-QH} and Kontsevich, Katzarkov and Pantev in
\cite{Katzarkov-Pantev-Kontsevich-ncVHS}) and so on, and used to express different aspects of mirror
symmetry. Another important aspect of the theory concerns the 
study of the functoriality of Gromov-Witten invariants via crepant resolutions or flop transitions in
terms of these structures (see \cite{MR2234886}, \cite{MR2360646}, \cite{CCIT-wall-crossingI},
\cite{CCIT-computingGW}, \cite{Bryan-Graber-2009}, \cite{MR2683208}, \cite{2013arXiv1309.4438B},
\cite{2014arXiv1407.2571B}, \cite{2014arXiv1410.0024C}, etc).\\

We first recall the classical construction of these invariants. Let $X$ be a smooth projective
variety (or orbifold). The basic ingredient to define GW-invariants is the moduli stack of stable
maps to $X$, denoted by $\overline{\mathcal{M}}_{g,n}(X,\beta)$, with a fixed degree
$\beta \in H_{2}(X,\mathbb{Z})$ \footnote{The (co)homology in this paper are the singular
  ones.}. The evaluation at the marked points gives maps of stacks
$\ev_i :\overline{\mathcal{M}}_{g,n}(X,\beta) \to X$ and forgetting the morphism and stabilising the
curve gives a map $p:\overline{\mathcal{M}}_{g,n}(X,\beta) \to \overline{\mathcal{M}}_{g,n}$ (See
Remark \ref{rem:2,morphisms}).

To construct the invariants, we integrate over ``the fundamental class'' of the moduli stack
$\overline{\mathcal{M}}_{g,n}(X,\beta)$. For this integration to be possible, we need this moduli
stack to be proper, which was proved by Behrend-Manin \cite{Behrend-Manin-stack-stable-mapGWI-1996}
and some form of smoothness. In general, the stack $\overline{\mathcal{M}}_{g,n}(X,\beta)$ is not
smooth and has many components with different dimensions. Nevertheless and thanks to a theorem of
Kontsevich \cite{MR1363062}, it is quasi-smooth - in the sense that locally it looks like the
intersection of two smooth sub-schemes inside an ambient smooth scheme. In genus zero however this
stack is known to be smooth under some assumptions on the geometry of $X$, for instance, when $X$ is
the projective space or a Grassmaniann, or more generally when $X$ is convex, \textit{i.e.}, if for
any map $f:\mathbb{P}^1\to X$, the group $\mathrm{H}^1(\mathbb{P}^1, f^*(\mathrm{T}_X))$
vanishes. See \cite{MR1492534}.

This quasi-smoothness has been used by Behrend-Fantechi to define in \cite{MR1437495} a
``virtual fundamental class'', denoted by $[\overline{\mathcal{M}}_{g,n}(X,\beta)]^{\vir}$, which is a cycle in the Chow ring of
$\overline{\mathcal{M}}_{g,n}(X,\beta)$ that plays the role of the usual fundamental class.

One of the most important result of Gromov-Witten invariants is that they form a cohomological field
theory, that is, there exist a family of morphisms 
\begin{align}
  \label{eq:2}
 I_{g,n,\beta}^{X}: H^{*}(X)^{\otimes n}& \to H^{*}(\overline{\mathcal{M}}_{g,n}) \\
(\alpha_{1}\otimes\ldots \otimes \alpha_{n}) &\mapsto
\mathrm{Stb}_{*}\left(\left[\overline{\mathcal{M}}_{g,n}(X,\beta)\right]^{\vir}\cup (\cup_i ev^{*}_i(\alpha_{i})) \nonumber
\right)
\end{align}
that satisfy some properties. Another formulation of this result is that we have a morphism of
operads between $\left(H_{*}(\overline{\mathcal{M}}_{g,n})\right)_{n\in \mathbb{N}}$ and the
endomorphism operad $\End(H^{*}(X))$ (see Corollary \ref{cor:alg,coho}). Yet a more concise way to
explain this, is to say that $H^{*}(X)$ owns a structure of algebra over the operads
$H_{*}(\overline{\mathcal{M}}_{g,n})$.

The main result of \cite{2015arXiv150502964M} is that it is possible to remove (co)homology from the previous statement.
The main result of \cite{2015arXiv150502964M} is the following
\begin{thm}[See Theorem \ref{thm,main}]
  Let $X$ be a smooth projective variety.

 The diagrams \begin{align*}
  \xymatrix{&\coprod_{\beta}\mathbb{R}\overline{\mathcal{M}}_{0,n+1}(X,\beta) \ar[ld]_{p,e_{1}, \ldots ,e_{n}}
              \ar[rd]^{p,e_{n+1}} & \\ \overline{\mathcal{M}}_{0,n+1}\times X^{n} &&
    \overline{\mathcal{M}}_{0,n+1}\times X
}
\end{align*}
 give a family of morphisms
   \begin{align*}
    \varphi_{n}:\overline{\mathcal{M}}_{0,n+1}  \to \underline{\End}^{\corr}(X)[n]:=\underline{\Hom}^{\corr}(X^{n},X)
  \end{align*}
that forms a lax morphism of $\infty$-operads in the category of derived stacks.
\end{thm}

We restrict our work to genus $0$ because we lack fundamental aspects for $\infty$-modular operads.

In this survey we omit the technical details and we insist on the ideas behind the
theorem. Nevertheless, we add some new statements with respect to \cite{2015arXiv150502964M} as Theorem
\ref{thm:colim} and Theorem \ref{thm:orientation} with the proofs given in the appendices.

\textbf{Acknowledgements:} We want to thank Bertrand To\"en for the organisation of the \'Etat de la
Recherche and also for some ideas to prove Theorem
\ref{thm:colim}. The first author thanks Daniel Naie who explains how to make these figures.

\section{Moduli space of stable maps, cohomological field theory and operads}
\label{sec:stable-maps}

In this section, we recall some notions and ideas related to Gromov-Witten theory. Most of them are in the book
of Cox-Katz \cite{Cox-Katz-Mirror-Symmetry}. The mathematical story started with the paper of
Kontsevich \cite{MR1363062} (see also Kontsevich-Manin \cite{MR1369420}) and was followed by many
more and interesting questions that we will skip here.

\subsection{Moduli space of stable maps}
\label{sec:moduli-space-stable-1}

Let $X$ be a smooth projective variety over $\mathbb{C}$. Let $\beta\in H_{2}(X,\mathbb{Z})$. Let
$g,n \in \mathbb{N}$.  Denote by $\Aff$ the category of affine scheme and by $\Grps$ the category of
groupoids.  We define the moduli space of stable maps by the following functor:
\begin{align*}
  \overline{\mathcal{M}}_{g,n}(X,\beta) : \Aff^{op}\longrightarrow & \Grps
\end{align*}
where $\overline{\mathcal{M}}_{g,n}(X,\beta)(S)$ is the following groupoids.  Objects are flat
proper morphisms $\pi:\mathcal{C}\to S$ together with $n$-sections $\sigma_{i}:S \to \mathcal{C}$
and a morphism $f:\mathcal{C}\to X$ such that for any geometric point $s \in S$, we have
\begin{enumerate}
\item the fiber $\mathcal{C}_{s}$ is a connected nodal curve of genus $g$ with $n$ distinct marked
  points which live on the smooth locus of $\mathcal{C}_{s}$.
\item  $f_{s}:\mathcal{C}_{s}\to X$ is of degree $\beta$, meaning $f_{*}[\mathcal{C}_{s}]=\beta$.
\item\label{item:stab} the automorphism group of $\Aut(\mathcal{C},\underline{\sigma},f)$ is finite
  where we denote $\underline{\sigma}=(\sigma_{1}, \ldots ,\sigma_{n})$. This condition is called
  \textit{stability} condition.
\end{enumerate}

For any affine scheme $S$, the morphism in the groupoid $\overline{\mathcal{M}}_{g,n}(X,\beta)(S)$
ar the isomorphisms $\varphi:\mathcal{C}\to \mathcal{C'}$ such that the following diagram is
commutative:
\begin{displaymath}
  \xymatrix{\mathcal{C} \ar[rdd]_-{\pi} \ar[rr]^-{\varphi}_{\sim} \ar[rd]^-{f} && \mathcal{C'} \ar[ldd]^-{\pi'}
    \ar[ld]_-{f'} \\ & X& \\ &S \ar@/^16pt/[uul]^{\sigma_{i}} \ar@/_16pt/[uur]_{\sigma'_{i}}&}
\end{displaymath}

Let $\varphi:S\to S'$ be a morphism of affine schemes. Let $(\mathcal{C}\to S,\underline{\sigma},f)$
be an object in $\overline{\mathcal{M}}_{g,n}(X,\beta)(S)$, then the pullback family defined by the
diagram below satisfies the three conditions above that is it is in
$\overline{\mathcal{M}}_{g,n}(X,\beta)(S')$
\begin{displaymath}
  \xymatrix{\mathcal{C'}\times_{S'}S \ar[dd] \ar[rr]^{\widetilde{\varphi}}
    \ar[rd]^-{\widetilde{\varphi}\circ f'} && \mathcal{C'} \ar[dd]^-{\pi'}
    \ar[ld]_-{f'} \\ & X& \\ S \ar@/^16pt/[uu]^{\varphi^{*}\sigma_{i}'}\ar[rr]^{{\varphi}}&& S' \ar@/_16pt/[uu]_{\sigma'_{i}}}
\end{displaymath}
Notice that the condition $(1), (2)$ a,d $(3)$ are stable by pull-back.

\begin{remark}
  Let explain the stability condition \eqref{item:stab} in more concrete terms (See \cite[\S 7.1.1
  p. 169]{Cox-Katz-Mirror-Symmetry}). Denote by $\mathcal{C}_{s,i}$ the irreducible components of
  $\mathcal{C}_{s}$ and by $f_{s,i}:\mathcal{C}_{s,i}\to X$ the restrictions of the morphism. Denote
  by $\beta_{i}=(f_{s,i})_{*}[\mathcal{C}_{s,i}]\in H_{2}(X,\mathbb{Z})$ the degree of $f_{s}$ on
  each irreducible component $\mathcal{C}_{s,i}$. On the irreducible component $\mathcal{C}_{s,i}$,
  a point is called \textit{special} if it is a nodal point or a marked point. The stability
  condition \eqref{item:stab} is equivalent to the following condition on each irreducible component
  : if $\beta_{i}=0$ and the genus of $\mathcal{C}_{s,i}$ is $0$ (resp. $1$) then
  $\mathcal{C}_{s,i}$ should have at least $3$ (resp. $1$) special points. So for example if
  $\beta_{i}\neq 0$ or the genus is greater than $2$ there is no condition on $\mathcal{C}_{s,i}$.
\end{remark}

In this text, we will never use the coarse moduli space of $\overline{\mathcal{M}}_{g,n}(X,\beta)$,
so all the morphisms that we will use are morphisms of stacks.

 \begin{ex}\label{sec:moduli-space-stable}
   Let us give an example in genus $0$ (see Figure \ref{fig:stab,maps}). Consider the following
   stable map in $\overline{\mathcal{M}}_{0,5}(X,\beta)$. All the $C_{i}$ are isomorphic to
   $\mathbb{P}^{1}$. The stability condition on this stable map imposes only that $\beta_{2}\neq 0$
   because $C_{2}$ has only $2$ special points.

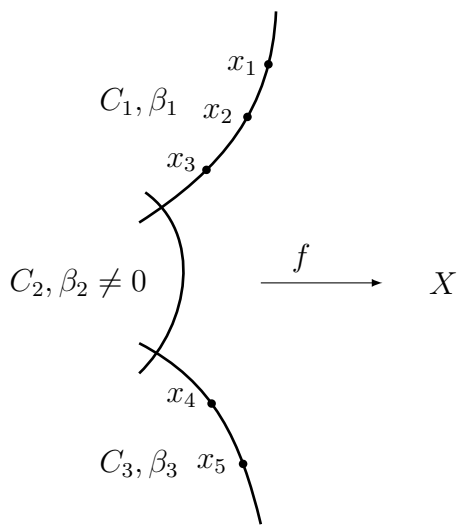
\begin{figure}[ht]
  \centering
  \begin{tikzpicture}[scale=0.8]
    \ThreeCurveVex{0}{0}{2.25}{3.5}{firstCurve} \ZeroCurveVex{0}{-2.5}{.1}{3}{secondCurve}
    \TwoCurveCave{0}{-2}{2}{-3}{thirdCurve} \draw (0,2) node {$C_{1},\beta_{1}$}; \draw (-1,-1) node
    {$C_{2},\beta_{2}\neq 0$}; \draw (0,-4) node {$C_{3},\beta_{3}$}; \draw [>=latex,->,black]
    (2,-1)-- (4,-1) node [above left, midway] {$f$}; \draw (5,-1) node {$X$}; \draw (1.7,2.6) node
    {$x_{1}$}; \draw (1.3,1.8) node {$x_{2}$}; \draw (0.7,1) node {$x_{3}$}; \draw (0.7,-2.9) node
    {$x_{4}$}; \draw (1.2,-4) node {$x_{5}$};
  \end{tikzpicture}
  \caption{Example of a stable map} \label{fig:stab,maps}
\end{figure}

\end{ex}

In particular, the moduli space of stable curve, denoted by $\overline{\mathcal{M}}_{g,n}$ is
$\overline{\mathcal{M}}_{g,n}(\pt,\beta=0)$. Notice that for $(g,n)\in\{(0,0),(0,1),(0,2),(1,0)\}$
the moduli space $\overline{\mathcal{M}}_{g,n}$ is empty.

\begin{remark}\label{rem:2,morphisms}
  There are two kinds of natural morphisms of stacks from the moduli space of stable maps.
  \begin{enumerate}
  \item For any $i\in\{1, \ldots ,n\}$, the evaluation morphism
    $e_{i}:\overline{\mathcal{M}}_{g,n}(X,\beta) \to X$ is the evaluation at the $i$-th marked point
    i.e., it sends the geometric point $(C,x_{1}, \ldots ,x_{n},f)$ to $f(x_{i})$.
  \item When $\overline{\mathcal{M}}_{g,n}$ is not empty, we define the morphism of stacks
    $p:\overline{\mathcal{M}}_{g,n}(X,\beta)\to \overline{\mathcal{M}}_{g,n}$ that forgets the map
    and stabilises the curve that is it sends $(C,x_{1}, \ldots ,x_{n},f)$ to
    $(C^{\stab},x_{1}, \ldots ,x_{n})$ where $C^{\stab}$ is obtained from $C$ by contracting all the
    unstable components (see \cite{Knudsen-moduli-stable-curves-II-1983} for the techniques). On the
    stable map of the example \ref{sec:moduli-space-stable}, forgetting the map $f$, the irreducible
    component $C_{2}$ become unstable (because it has only $2$ special points). So the image by $p$
    is the following stable curve (see Figure \ref{fig:stab}).

\begin{figure}[ht]
  \centering
  \begin{tikzpicture}[scale=0.7]
    \ThreeCurveVex{0}{0}{2.25}{3.5}{firstCurve} \TwoCurveCave{0}{0.5}{2}{-3}{thirdCurve} \draw
    (0.2,2) node {$C_{1}$}; \draw (0.1,-1.3) node {$C_{3}$};

    \draw (1.7,2.6) node {$x_{1}$}; \draw (1.3,1.8) node {$x_{2}$}; \draw (0.7,1) node {$x_{3}$};
    \draw (0.7,-0.7) node {$x_{4}$}; \draw (1.2,-1.7) node {$x_{5}$};
  \end{tikzpicture}
  \caption{The stabilisation of the stable maps of Figure \ref{fig:stab,maps}} \label{fig:stab}
\end{figure}
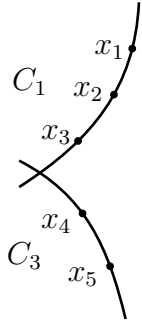
\end{enumerate}
\end{remark}

\begin{thm}[Deligne-Mumford \cite{MR0262240}, Kontsevich-Manin \cite{MR1369420}, Behrend-Fantechi
  \cite{MR1437495}]
  \
  \begin{enumerate}
  \item The moduli space $\overline{\mathcal{M}}_{g,n}$ is a proper smooth Deligne-Mumford stack of
    dimension $3g-3+n$.
  \item The moduli space $\overline{\mathcal{M}}_{g,n}(X,\beta)$ is a proper (not smooth in general)
    Deligne-Mumford stack. It has an expected dimension (see remark below for the meaning) which is
    \begin{displaymath}
      \int_{\beta}c_{1}(TX) +(1-g)\dim X +3g-3+n
    \end{displaymath}
  \item There exists a class, denoted by $[\overline{\mathcal{M}}_{g,n}(X,\beta)]^{\vir}$, in the
    Chow ring $A_{*}(\overline{\mathcal{M}}_{g,n}(X,\beta))$ of degree equal to the expected
    dimension of $\overline{\mathcal{M}}_{g,n}(X,\beta)$ which satisfies some functorial properties.
  \end{enumerate}
\end{thm}

\begin{remark}
  \begin{enumerate}
  \item To use standard tools of intersection on the moduli space of stable maps we need this moduli
    space to be proper and smooth. The smoothness would give us the existence of a well-defined
    fundamental class. Nevertheless, the moduli space of stable maps
    $\overline{\mathcal{M}}_{g,n}(X,\beta)$, which is not smooth in general, could have different
    irreducible components of different dimensions with some very bad singularities. So the problem
    is to define an ersatz of a fundamental class. This was done by Behrend-Fantechi in
    \cite{MR1437495} where they defined the \textit{virtual fundamental class} (see
    \S~\ref{sec:comp-theor-two}).
  \item In some very specific case the moduli space of maps is smooth : for example only in genus
    $0$ for homogeneous variety like $\mathbb{P}^{n}$, grassmannian or flag varieties. In these
    cases, the virtual dimension is the actual dimension and the virtual fundamental class is the
    fundamental class.
  \item The computation of the expected dimension comes from deformation theory. Namely, a
    deformation of a stable maps turns to be a deformation of the underlying curve plus a
    deformation of the map.  As $\overline{\mathcal{M}}_{g,n}$ is smooth, the deformation functor of
    the curve has no obstruction and the tangent space has the dimension of
    $\overline{\mathcal{M}}_{g,n}$ which is $3g-3+n$. For the maps, the deformation functor has a
    non zero obstruction. More precisely, at a point
    $(C,\underline{x},f) \in \overline{\mathcal{M}}_{g,n}(X,\beta)$, the tangent space is
    $H^{0}(C,f^{*}TX)$ and an obstruction is $H^{1}(C,f^{*}TX)$. Making this in family, one gets two
    quasi coherent sheaves that are not vector bundles. Nevertheless the Euler characteristic can be
    computed via the Hirzebruch-Riemann-Roch theorem:
    \begin{displaymath}
      \chi(C,f^{*}TX)=\dim H^{0}(C,f^{*}TX)-\dim H^{1}(C,f^{*}TX)=\int_{C}\Td(TC)\ch(f^{*}TX)
    \end{displaymath}
    is constant and equals to $\int_{\beta}c_{1}(TX) +(1-g)\dim X$.
  \end{enumerate}
\end{remark}

We will now introduce another moduli space which was introduce by Costello \cite{MR2247968} and
which will play a crucial role latter. Let $\NE(X)$ be the subset of $H_{2}(X,\mathbb{Z})$ of
classes given by the image of a curve i.e. the subset of all $f_{*}[C]$ for any morphism $f:C\to X$.
Let define $\mathfrak{M}_{g,n\beta}$ as the moduli space of nodal curve of genus $g$ with $n$ marked
smooth points where each irreducible component $C_{i}$ has a labelled $\beta_{i}$ (notice that this
$\beta_{i}$ is not the degree of a map because there is no map from $C\to X$, it is just a
labbeled. At the end of the day, it will be related to the degree of a map but not here) such that
\begin{itemize}
\item $\sum_{i}\beta_{i}=\beta$
\item if $\beta_{i}=0$ then $C_{i}$ is stable i.e., if $C_{i}$ is of genus 0 then it has at least
  $3$ special points and if the genus is $1$ then it has at least $1$ special point.
\end{itemize}

We have a natural morphism of stacks $p:\mathfrak{M}_{g,n+1,\beta}\to \mathfrak{M}_{g,n,\beta}$
which forgets the $(n+1)-th$ marked point and contracts the irreducible components that are not
stable.

\begin{thm}[\cite{MR2247968}]\label{thm,Costello}
  \begin{enumerate}
  \item The stack $\mathfrak{M}_{g,n,\beta}$ is a smooth Artin stack.
  \item The morphism $p:\mathfrak{M}_{g,n+1,\beta}\to \mathfrak{M}_{g,n,\beta}$ is the universal
    curve.
  \end{enumerate}
\end{thm}

\begin{remark}
  \begin{enumerate}
  \item Notice that forgetting the last marked point and contracting the unstable component gives a
    morphism $\overline{\mathcal{M}}_{g,n+1} \to \overline{\mathcal{M}}_{g,n}$ which is also the
    universal curve (See \cite{Knudsen-moduli-stable-curves-II-1983}).
  \item The Artin stack of prestable\footnote{where we do not ask any stability condition on
      irreducible components see \cite[p.179]{Cox-Katz-Mirror-Symmetry}.} curves, denoted by
    $\mathfrak{M}^{pre}_{g,n}$ also have a universal curve which is not
    $\mathfrak{M}^{\pres}_{g,n+1}$. As there is no stability condition on the moduli space of
    prestable curves, forgetting a marked point never contract a rational curve. So forgetting a
    marked point $\mathfrak{M}^{\pres}_{g,n+1}\to \mathfrak{M}^{\pres}_{g,n}$ is not the universal
    curve. 
  \end{enumerate}
\end{remark}

Let us explain the meaning of being an universal curve of $\mathfrak{M}_{g,n,\beta}$.  Let $C$ be a
curve of genus $g$ with $4$ marked points with a label $\beta$.  This is equivalent by definition to
a morphism $\pt \to \mathfrak{M}_{g,4,\beta}$. Being a universal curve means that we have the
$C=\mathfrak{M}_{g,5,\beta}\times_{\mathfrak{M}_{g,4,\beta}}\pt$ that is the following diagram
\begin{displaymath}
  \xymatrix{ C \ar[r]^{\varphi} \ar[d]&\mathfrak{M}_{g,5,\beta}\ar[d]\\ \pt \ar[r]&\mathfrak{M}_{g,4,\beta}}
\end{displaymath}
is cartesian.  Let explain the morphism $\varphi$. To a smooth point
$y\in C\setminus \{x_{1}, \ldots ,x_{4}\}$, $f(y)$ is the curve $C$ where $y$ is now $x_{5}$.  If
$y=x_{i}$, then $\varphi(y)$ is the curve $C$ where we attach a $\mathbb{P}^{1}$ at $x_{i}$ (let's
say at $0$ of this $\mathbb{P}^{1}$) with $\beta=0$ and you marked $x_{i}$ and $x_{5}$ at $1$ and
$\infty$. If $y$ is a node which is the intersection with $C_{i}$ and $C_{j}$, then we replace the
node by a $\mathbb{P}^{1}$ with degree $0$ which meet $C_{i}$ at $0$, $C_{j}$ at $\infty$ and we
marked the point $1$ by $x_{5}$ on this $\mathbb{P}^{1}$.
 
Here is a picture that we hope makes this clearer (see Figure \ref{fig:universal}).  Forgetting the
last point makes the component $(\mathbb{P}^{1},\beta=0)$ unstable so one should contract it and we
get back $C$.

    




\begin{figure}[ht]
  \centering
 
  \begin{tikzpicture}[scale=1]
    \FourCurveVex{0}{0}{0.1}{4}{firstCurve}; \draw (0.3,1) node {$x_{1}$}; \draw (0.5,1.7) node
    {$x_{2}$}; \draw (0.6,2.5) node {$x_{3}$}; \draw (0.4,3.2) node {$x_{4}$}; \draw (1.3,3.2) node
    {$y$}; \draw (1.1,1.) node {$y$}; \draw (-0.3,2) node {$C,\beta$};

    \draw [-latex] (2,2) -- (5,2) node [above,midway] {$f$} ; \draw (6,2) node
    {$\mathfrak{M}_{g,5,\beta}$};




    \FourCurveVexSmall{3}{-3}{0.1}{2}{firstCurve} \draw (1.5,-2) node {$f(y)=$}; \draw (3,-2.6) node
    {$x_{1}$}; \draw (3.1,-2.1) node {$x_{2}$}; \draw (3.1,-1.7) node {$x_{3}$}; \draw (3.1,-1.4)
    node {$x_{4}$}; \draw (4.2,-2.4) node {$y=x_{5}$}; \draw (8.2,-2) node { if $y$ is not a marked
      point};

    \ThreeCurveVex{3}{-7}{.1}{2}{firstCurve} \TwoCurveCaveVert{3}{-5.5}{2}{0.1}{thirdCurve} \draw
    (2,-6) node {$f(y)=$}; \draw (3.7,-6.5) node {$x_{1}$}; \draw (3.8,-6.1) node {$x_{2}$}; \draw
    (3.8,-5.5) node {$x_{3}$}; \draw (3.6,-4.7) node {$x_{4}$}; \draw (4.7,-4.7) node {$y=x_{5}$};
    \draw (3,-7.5) node {$C,\beta$}; \draw (6,-5.4) node {$\mathbb{P}^{1},\beta=0$}; \draw
    (8.2,-6.5) node { if $y$ is the marked point $x_{4}$};

  \end{tikzpicture}
  \caption{Universal curve}\label{fig:universal}
\end{figure}
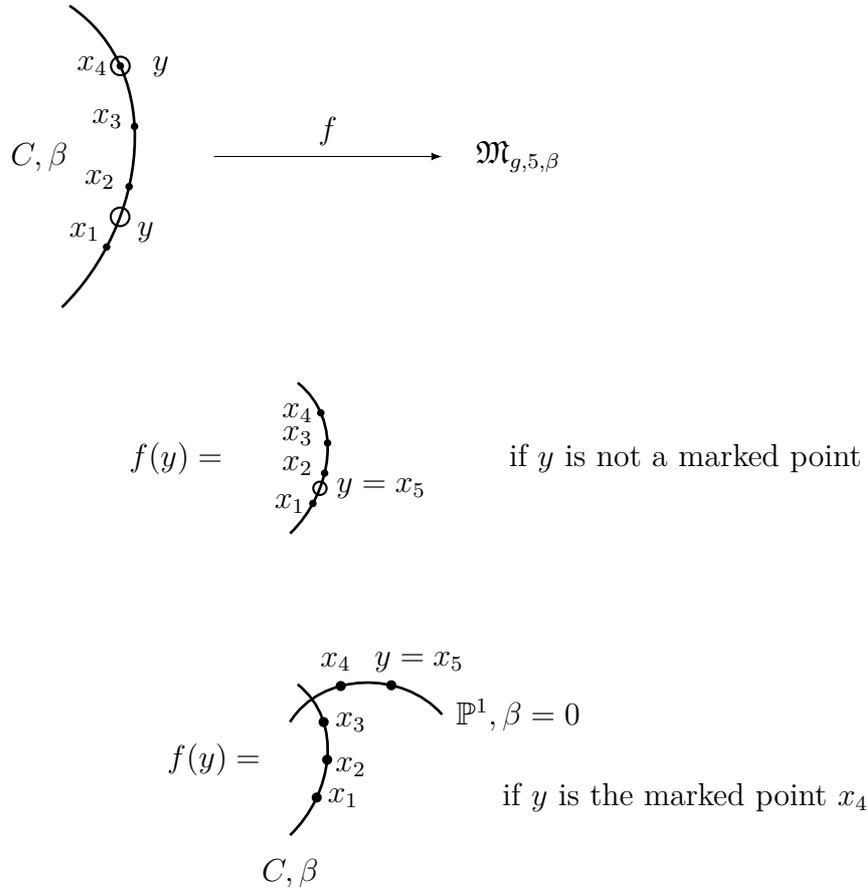

\subsection{Gromov-Witten classes and cohomological field theory}
\label{sec:grom-witt-class}
We first define the Gromov-Witten classes. Let $\alpha_{1}, \ldots ,\alpha_{n} \in H^{*}(X)$. Let
$\beta\in H_{2}(X,\mathbb{Z})$. We define the following morphism
\begin{align*}
  \varphi_{g,n,\beta}:H^{*}(X)\times \cdots \times H^{*}(X)& \longrightarrow
  H^{*}(\overline{\mathcal{M}}_{g,n}) \\ (\alpha_{1}, \ldots ,\alpha_{n})& \longmapsto
  p_{*}\left(\prod_{i=1}^{n}e_{i}^{*}\alpha_{i}\cap [\overline{\mathcal{M}}_{g,n}(X,\beta)]^{\vir}\right)
\end{align*}

\begin{thm}[Kontsevich-Manin \cite{MR1369420}]
All these maps $\{\varphi_{g,n,\beta}\}_{g,n\in\mathbb{N},\beta\in H_{2}(X,\mathbb{Z})}$ together form a
cohomological field theory.  
\end{thm}

\begin{remark}
  \begin{enumerate}
  \item We refer to \cite{MR1369420} for a complete
    definition of a cohomological field theory.
  \item Unwindy the definition, is the so-called splitting property. Let
    $g_{1},g_{2},n_{1},n_{2} \in \mathbb{N}$. Denote by $g=g_{1}+g_{2}$ and $n= n_{1}+n_{2}$. Consider
    the gluing morphism of stacks
    \begin{align}\label{eq:gluing,moprh}
 g:     \overline{\mathcal{M}}_{g_{1},n_{1}+1}\times \overline{\mathcal{M}}_{g_{2},n_{2}+1} &\to
      \overline{\mathcal{M}}_{g,n} \\ (C_{1},C_{2}) & \longmapsto C_{1}\circ C_{2} \nonumber
    \end{align}
    that identifies the $n_{2}+1$-th marked point of $C_{2}$ with the first marked point of
    $C_{1}$. Notice that the gluing morphism above is given by the pushout.  More
    precisely, let $({C}_{1}\to S,\underline{\sigma})$ in
    $\overline{\mathcal{M}}_{g_{1},n_{1}+1}(S)$ and $(C_{2}\to S,\underline{\sigma})$ in
    $\overline{\mathcal{M}}_{g_{2},n_{2}+1}(S)$ then $ C_{1}\circ C_{2}$ is the
    pushout\footnote{Notice that pushouts do not exist for any morphisms of schemes in the category
      of schemes but pushout along closed immersion does exist.} $C_{1}\coprod_{S}C_{2}$ given by
    the two closed immersion given by the marking $\sigma_{1}:S\to C_{1}$ and $\sigma_{n_{2}+1}:S\to
    C_{2}$.

This corresponds to the following picture
  \begin{figure}[ht]
     \centering
     \begin{tikzpicture}
        \FourCurveVexNorm{0}{0.5}{2}{3}{firstCurve}
\draw (0.2,2.5) node {$C_{1}$}; 
\draw (1.5,2.6) node {$x_{4}$}; 
\draw (1.3,2.1) node {$x_{3}$}; 
\draw (0.8,1.5) node {$x_{2}$}; 
\draw (0.4,1.1) node {$x_{1}$}; 

\ThreeCurveCave{0}{-.5}{2}{-3}{thirdCurve}
\draw (0.5,-2.3) node {$C_{2}$}; 

\draw (0.7,-1.4) node {$x_{3}$}; 
\draw (1.2,-2.1) node {$x_{2}$}; 
\draw (1.4,-2.8) node {$x_{1}$}; 

\draw [->]  (1, 0) -- (3,0) node [above, midway] {$g$};

\ThreeCurveVexNorm{4}{-0.5}{2}{3}{firstCurve}
\draw (5.5,1.9) node {$x_{5}$}; 
\draw (5.3,1.2) node {$x_{4}$}; 
\draw (4.8,.5) node {$x_{3}$}; 

\TwoCurveCave{4}{.5}{2}{-3}{thirdCurve}
\draw (4.8,-.6) node {$x_{2}$}; 
\draw (5.4,-1.7) node {$x_{1}$}; 

\draw (7,0) node {$C_{1}\circ C_{2}$}; 

\end{tikzpicture}
      \caption{Gluing curves: the output of $C_{2}$, that is $x_{3}$, with the first input of $C_{1}$}
     \label{fig:gluing}
   \end{figure}
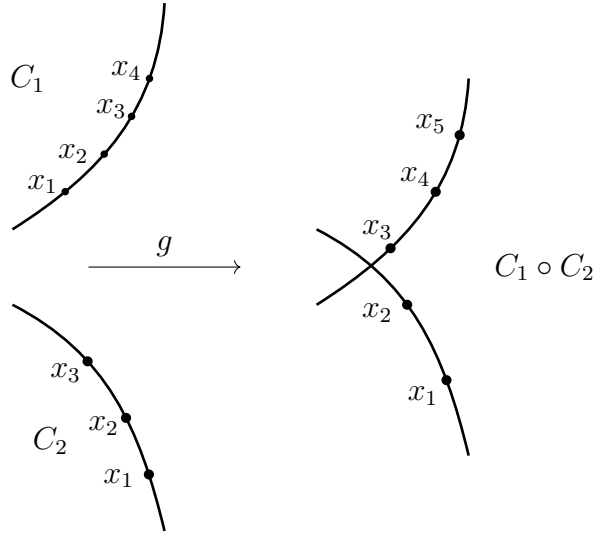

The splitting formula is the following
   \begin{align}\label{eq:split}
     g^{*}\varphi_{g,n,\beta}(\alpha_{1}, \ldots ,\alpha_{n})=\sum_{\stackrel{g_{1}+g_{2}=g}{\beta_{1}+\beta_{2}=\beta}}\sum_{a=0}^{s}\varphi_{g_{1},n_{1}+1,\beta_{1}}(\alpha_{1}, \ldots ,\alpha_{n_{1}},T_{a})\varphi_{g_{2},n_{2}+1,\beta_{2}}(T^{a},\alpha_{n_{1}+1}, \ldots ,\alpha_{n})
   \end{align}
where $(T_{a})_{a\in\{0, \ldots ,s\}}$ is a basis of $H^{*}(X)$ and $(T^{a})$ is its Poincaré dual basis.
  \end{enumerate}
Beyond  this formula, the idea is that we can control the behaviour of the virtual fondamental class when we glue
curves. We will see this again later.
\end{remark}

Restricting to genus $0$, we can reformulate this equality \eqref{eq:split} by the following
statement.

\begin{cor}\label{cor:alg,coho}
  We have a morphism of operads in vector spaces
  \begin{align*}
 \psi_{n,\beta}:   H_{*}(\overline{\mathcal{M}}_{0,n+1}) \to \End(H^{*}(X))[n]:=\Hom(H^{*}(X)^{\otimes n}, H^{*}(X))
  \end{align*}
given by
\begin{align*}
  \psi_{0,n,\beta}(\gamma)(\alpha_{1}, \ldots
  ,\alpha_{n})=(e_{n+1})_{*}\left(p^{*}\gamma\cup \prod_{i=1}^{n}e_{i}^{*}\alpha_{i}\cap [\overline{\mathcal{M}}_{0,n}(X,\beta)]^{\vir}\right)
\end{align*}
\end{cor}
Another way of expressing exactly the same statement is to say that the cohomology $H^{*}(X)$ is an
$\{H_{*}(\overline{\mathcal{M}}_{0,n+1})\}_{n\geq 2}$-algebra.  The goal of this survey is to
explain how to remove
the (co)homology from this corollary and doing this at the geometrical level.

\subsection{Reviewed on operads}
\label{sec:recall-operads}

We add this section for completeness as operads are not so well known to algebraic
geometer\footnote{The first author did not know this notion before the working seminar in
  Montpellier where these ideas were first discussed.}.

An operad is the following data :
\begin{enumerate}
\item A family of objects in a category (vector spaces, schemes or Deligne-Mumford stacks)
$\mathcal{O}(n)$ for all $n\in \mathbb{N}$.  The example that one should have in mind for this note is
$\mathcal{O}(n)=\overline{\mathcal{M}}_{0,n+1}$. We should think that $\mathcal{O}(n)$ as a
collection of operations, each with $n$ inputs and one output. In the case of $\overline{\mathcal{M}}_{0,n+1}$, the
marked points $x_{1}, \ldots ,x_{n}$ can be thought as the inputs and the last marked points,
$x_{n+1}$, is thought as the output.
\item A collection  of operations: putting the output of $\mathcal{O}(b)$ with the $i$-th input of $\mathcal{O}(a)$. Let $a,b \in \mathbb{N}$, for any
$i\in\{1, \ldots ,a\}$, we have
\begin{align}\label{eq:5}
\circ_{i} :  \mathcal{O}(a)\times \mathcal{O}(b) &\to \mathcal{O}(a+b-1) 
\end{align}
satisfying some relations like associativity of the compositions.
\end{enumerate}

\begin{example}We give three examples of operads that we will use in the next sections.
  \begin{enumerate}
  \item The example $\mathcal{O}(n)=\overline{\mathcal{M}}_{0,n+1}$ is an operads in DM stacks where the
    composition $C_{1}\circ_{i}C_{2}$ is obtained by gluing the last marked point of $C_{2}$ to the
$i$-th marked point of $C_{1}$ (see \eqref{eq:gluing,moprh} and Figure \ref{fig:gluing} for an example of $\circ_{1}$ with stable
curves). Notice that here $\mathcal{O}(0)$ and $\mathcal{O}(1)$ are empty. A standard way of
completing this is to put $\mathcal{O}(0)=\mathcal{O}(1)=\pt$ so that $\mathcal{O}(1)$ is the unit.
\item Another example of operads that we will use is $\mathcal{O}_{\beta}(n)=\mathfrak{M}_{0,n+1,\beta}$. This is a
  graded operad that is in the composition \eqref{eq:5}, we sum the grading :
  \begin{displaymath}
    \circ_{i} :  \mathcal{O}_{\beta}(a)\times \mathcal{O}_{\beta'}(b) \to \mathcal{O}_{\beta+\beta'}(a+b-1) 
  \end{displaymath}
The composition morphism for this operad is by  gluing the curves as in the previous example.
\item Let $V$ a vector space. Put $\mathcal{O}(n)=\End(V)[n]:=\Hom(V\times\cdots\times V,V)$. This is
  called the endomorphism operad in vector spaces. The composition is given by
  \begin{displaymath}
    (f \circ_{i} g)(v_{1}, \ldots ,v_{a+b-1})=f(v_{1}, \ldots ,v_{i-1},g(v_{i}, \ldots ,v_{i+b-1}),v_{i+b} \ldots ,v_{a+b-1}) 
  \end{displaymath}
  \end{enumerate}
\end{example}

Let $\mathcal{O}:=\{\mathcal{O}(n)\}_{n\in\mathbb{N}}$ and $\mathcal{E}:=\{\mathcal{E}(n)\}_{n\in\mathbb{N}}$ be two
operads. A morphism of operads from $f:\mathcal{O}\to\mathcal{E}$ is a family of morphism
$f_{n}:\mathcal{O}(n)\to \mathcal{E}(n)$ such that the following diagram is commutative
\begin{align}\label{diag:morph,operad}
  \xymatrix{\mathcal{O}(a)\times \mathcal{O}(b)\ar[r]^{f_{a},f_{b}}\ar[d]_{\circ_{i}}&\mathcal{E}(a)\times \mathcal{E}(b) \ar[d]^{\circ_{i}}\\
\mathcal{O}(a+b-1) \ar[r]^{f_{a+b-1}}& \mathcal{E}(a+b-1)}
\end{align}


\section{Lax algebra structure on $X$}
\label{sec:lax-algebra-struct}

In Corollary \ref{cor:alg,coho}, we have a collection of morphisms
$H_{*}(\overline{\mathcal{M}}_{0,n+1}) \to \End(H^{*}(X))[n]$ that form a morphism of operads. The
idea is to remove the (co)homology from this statement, that is, to construct in  a purely
geometrical way, a collection morphisms
$\overline{\mathcal{M}}_{0,n+1} \to \End(X)[n]$ in an appropriate category and then to see if these
morphisms form a morphism of operads. The correct category is the $(\infty,1)$-category of derived
stacks and the morphism is only a lax morphism of $\infty$-operads (see Theorem \ref{thm,main}).
\subsection{Main result}
\label{sec:main-result}

Denote by $\mathbb{R}\overline{\mathcal{M}}_{0,n+1}(X,\beta)$ the derived enhancement of
$\overline{\mathcal{M}}_{0,n+1}(X,\beta)$ (see subsection \ref{sec:deriv,enhanc,under}).
From the two natural morphisms of Remark \ref{rem:2,morphisms}, we get the following diagram 
\begin{align}\label{key,diag}
  \xymatrix{&\coprod_{\beta}\mathbb{R}\overline{\mathcal{M}}_{0,n+1}(X,\beta) \ar[ld]_{p,e_{1}, \ldots ,e_{n}}
              \ar[rd]^{p,e_{n+1}} & \\ \overline{\mathcal{M}}_{0,n+1}\times X^{n} &&
    \overline{\mathcal{M}}_{0,n+1}\times X
}
\end{align}
We prefer to state our theorem and then give explanations about it.
\begin{thm}\label{thm,main}
  Let $X$ be a smooth projective variety. The diagram \eqref{key,diag} give a family of morphisms
   \begin{align*}
    \varphi_{n}:\overline{\mathcal{M}}_{0,n+1}  \to \underline{\End}^{\corr}(X)[n]:=\underline{\Hom}^{\corr}(X^{n},X)
  \end{align*}
that forms a lax morphism of $\infty$-operads in the category of derived stacks.
\end{thm}

\begin{remark}
In more conceptual terms,  $X$ is lax $\{\overline{\mathcal{M}}_{0,n+1}\}_{n}$-algebra in the
  category of correspondence in derived stack.
\end{remark}

In the next sections, we will explain the contents of this theorem, namely
\begin{itemize}
\item In \S \ref{sec:categ,corr}, we define the notion of correspondances in a cateogry.
\item In \S \ref{sec:deriv,enhanc,under},  we define the natural derived enhancement of the moduli space of
  stable maps $\overline{\mathcal{M}}_{g,n}(X,\beta)$ and in \ref{sec:definition,End}, we explain
  the underlying notation $\underline{\Hom}^{\corr}(X^{n},X)$.
\item In \S \ref{sec:lax,morphism}, we explain what is a lax morphism between $\infty$-operads.
\item Th notion of $\infty$-operads is a bit delicat and it is explain in  In \S \ref{thm:brane}.
\end{itemize}

\subsection{Category of correspondances}
\label{sec:categ,corr}
Let $\dSt$ be the $\infty$-category of derived stacks. We denote $\dSt^{\corr}$ the
$(\infty,2)$-category of correspondences in derived stack which is defined informally as follows
(See \S 10 in \cite{1212.3563}). To
have a formal definition, we refer to the notion of span in the website nLab.
\begin{enumerate}
\item Object of $\dSt^{\corr}$ are objects of $\dSt$.
\item The $1$-morphism of $\dSt^{\corr}$ between $X$ and $Y$, denoted by $X\dasharrow Y$, is a diagram
  \begin{displaymath}
    \xymatrix{&U\ar[ld]_{g}\ar[rd]^{f}&\\ X&&Y}
  \end{displaymath}
There is no condition on $f$ or $g$. The composition is given by  fiber product
  \begin{displaymath}
    \xymatrix{&&U\times_{Y}V \ar[rd]\ar[ld]&&\\ &U\ar[ld]\ar[rd]&&V \ar[rd]\ar[dl]&\\ X&&Y&&Z}
  \end{displaymath}
Notice that a morphism from $X$ to $Y$ is also a morphism from $Y$ to $X$ but the composition is not
the identity which is 
\begin{displaymath}
    \xymatrix{&X\ar[ld]_{\Id}\ar[rd]^{\Id}&\\ X&&X}
  \end{displaymath}
Hence a morphism of scheme $f:X\to Y$ induces a morphism $X\dasharrow Y$ in correspondances given by $\Id_{X}:X\to X$
and $f:X\to Y$. This morphism $X\dasharrow Y$  is an isomorphism if and only if we have  $X=X\times_{Y}X$
i.e., $f$ is a monomorphism.
\item The $2$-morphisms are not necessarily  isomorphisms, they are $\alpha: U\to V$ that make the
  diagram commutative.
  \begin{displaymath}
     \xymatrix{&U\ar[dd]^{\alpha}\ar[ld]\ar[rd]&\\ X&&Y \\ &V\ar[ru]\ar[ul]&}
  \end{displaymath}
\end{enumerate}

The diagram 
  \begin{align}\label{key,diag,bis}
  \xymatrix{&\mathbb{R}\overline{\mathcal{M}}_{0,n+1}(X,\beta) \ar[ld]_{e_{1}, \ldots ,e_{n}}
              \ar[rd]^{e_{n+1}} & \\  X^{n} &&
 X
}
\end{align}
is by definition a morphism in $\dSt^{\corr}$ between $X^{n}\dasharrow X$. Notice that the object
that makes the correspondence is a derived stack so we need to be in the category of $\dSt^{\corr}$
and not in the category of correspondence in schemes (or Deligne-Mumford stacks).
\subsection{Derived enhancement}
\label{sec:deriv,enhanc,under}

\subsubsection{Derived enhancement of $\mathbb{R}\overline{\mathcal{M}}_{0,n+1}(X,\beta)$.}
\label{sec:deriv-enhanc-mathbbr}

Here we follow the idea of Sch\"urg-To\"en-Vezzosi \cite{MR3341464} with a small modification. Let $g,n \in\mathbb{N}$ and $\beta\in
H_{2}(X,\beta)$. Recall the definition of  $\mathfrak{M}_{g,n,\beta}$ the moduli space defined before Theorem
\ref{thm,Costello}. 
We denote the relative internal hom in derived stacks by
\begin{align}\label{eq:8}
  \mathbb{R}\Hom_{\dst/\mathfrak{M}_{g,n,\beta}}(\mathfrak{M}_{g,n+1,\beta}, X \times \mathfrak{M}_{g,n,\beta})
\end{align}
As $\mathfrak{M}_{g,n+1,\beta}\to \mathfrak{M}_{g,n,\beta}$ is the universal curve, a point in
$\mathbb{R}\Hom_{\dst/\mathfrak{M}_{g,n,\beta}}(\mathfrak{M}_{g,n+1,\beta}, X \times
\mathfrak{M}_{g,n,\beta})$ is by definition a morphism from $f:C\to X$ where $[C] \in
\mathfrak{M}_{g,n,\beta}$. Notice that the degree $f$ is not related for the moment to $\beta$.
The truncation of \eqref{eq:8} is
\begin{displaymath}
  \Hom_{\dst/\mathfrak{M}_{g,n,\beta}}(\mathfrak{M}_{g,n+1,\beta}, X \times \mathfrak{M}_{g,n,\beta})
\end{displaymath}
and inside it, we have an immersion
\begin{align}\label{eq:9}
 \overline{\mathcal{M}}_{g,n}(X,\beta) \hookrightarrow \Hom_{\dst/\mathfrak{M}_{g,n,\beta}}(\mathfrak{M}_{g,n+1,\beta}, X \times \mathfrak{M}_{g,n,\beta})
\end{align}
given by stable maps $(\mathcal{C},\underline{\sigma},f:\mathcal{C}\to X)$ such that the degree of
$f$ on each irreducible component $\mathcal{C}_{i}$ of $\mathcal{C}$, the degree of
$f\mid_{\mathcal{C}_{i}}$ is $\beta_{i}$ i.e., we have the equality
$(f\mid_{\mathcal{C}_{i}})_{*}[\mathcal{C}_{i}]=\beta_{i}$. This immersion is open because the
degree is discrete.

Using the following result of Schu\"rg-To\"en-Vezzosi, we have

\begin{prop}[Proposition 2.1 in \cite{MR3341464}] Let $X$ be in $\dSt$ and an open immersion of
  $Y\hookrightarrow t_{0}(X)$ where $t_{0}(X)$ is the truncation of $X$. Then there exists a unique derived enhancement of $Y$, denoted by $\widehat{Y}$, such that the
  following diagram is cartesian
  \begin{displaymath}
    \xymatrix{Y\ar@^{(->}[r]^{open} \ar@^{(->}[d]_{closed} & t_{0}(X) \ar@^{(->}[d]^{closed}\\\widehat{Y} \ar@^{(->}[r]^-{open}& X}
  \end{displaymath}
\end{prop}

Taking $Y=\overline{\mathcal{M}}_{g,n}(X,\beta)$ and the open immersion \eqref{eq:9}, we get a
derived  enhancement, which we
denote by $\mathbb{R}\overline{\mathcal{M}}_{g,n}(X,\beta)$.
\begin{remark}
To define the derived enhancement of the moduli space of stable maps
$\overline{\mathcal{M}}_{g,n}(X,\beta)$, Sch\"urg-To\"en-Vezzosi (see  \cite{MR3341464}) used the
moduli space of prestable curve denoted by $\mathfrak{M}^{\pres}_{g,n}$
instead of the moduli space of Costello $ \mathfrak{M}_{g,n,\beta}$ in \eqref{eq:8}. 
So they use the universal curve
of $\mathfrak{M}^{\pres}_{g,n}$ in \eqref{eq:8} instead of $\mathfrak{M}_{g,n+1,\beta}$.
As we will see in the proof (see section \ref{sec:proof-our-main}), the fact that
$\mathfrak{M}_{g,n+1,\beta}$ is the universal curve is fundamental, that is the reason why we made
this little change.

Notice that their derived enhancement is the same
as ours  as the morphism $\mathfrak{M}_{g,n,\beta}\to \mathfrak{M}_{g,n}$ is étale (See
\cite{MR2247968}). 
\end{remark}

\subsubsection{Definition of $\underline{\Hom}^{\corr}(X^{n},X)$}
\label{sec:definition,End}
The underling notation means the internal hom $\underline{\Hom}^{\corr}(X^{n},X)$. To be more precise,
it is the sheaf
\begin{displaymath}
\underline{\Hom}^{\corr}(X^{n},X)(U):=\Hom^{\corr}(X^{n}\times U, X\times U)
\end{displaymath}
It turns out that this is a derived stack because $\Hom^{\corr}(X^{n}\times U, X\times U)$ is the
same as the category of derived stack over $X^{n+1}\times U$. 

  By Yoneda's lemma, the morphism $\varphi_{n}$ of Theorem \ref{thm,main} is exactly given by an object in
  $\Hom^{\corr}(X^{n}\times \overline{\mathcal{M}}_{0,n+1}, X\times \overline{\mathcal{M}}_{0,n+1})$
  which is the diagram \eqref{key,diag}.

\subsection{Lax morphism}
\label{sec:lax,morphism}
Recall that a classical morphism of operad is a commutative diagram \eqref{diag:morph,operad}. A \textit{lax
morphism} is given by a collection of  $2$-morphisms $(\alpha_{a,b})_{a,b\in\mathbb{N}}$ which are not an isomorphism.
\begin{displaymath}
  \xymatrix{\mathcal{O}(a)\times
    \mathcal{O}(b)\ar[r]^{(f_{a},f_{b})}\ar[d]_{\circ_{i}}&\mathcal{E}(a)\times \mathcal{E}(b) \ar@{=>}[ld]_{\alpha_{a,b}}
    \ar[d]^{\circ_{i}}\\
\mathcal{O}(a+b-1) \ar[r]^{f_{a+b-1}}& \mathcal{E}(a+b-1)}
\end{displaymath}

In the following, we will explain why the Theorem 
\ref{thm,main} is lax in geometrical term . Let $\sigma\in
\overline{\mathcal{M}}_{0,a+1}$ and $\tau\in \overline{\mathcal{M}}_{0,b+1}$.
Denote by $\mathbb{R}\overline{\mathcal{M}}_{0,a+1}^{\sigma}(X,\beta)$
(resp. $\mathbb{R}\overline{\mathcal{M}}_{0,a+1}^{\tau}(X,\beta)$ )  the inverse image of
$p^{-1}(\sigma)$ (resp. $p^{-1}(\tau)$).

  The composition is $f_{a+b-1}\circ\circ_{i}$ given by
 \begin{align}
   \label{eq:11}
   \xymatrix{&\coprod_{\beta}\mathbb{R}\overline{\mathcal{M}}_{0,a+b+1}(X,\beta) \ar[rd] \ar[ld]&\\X^{a+b}&&X}
 \end{align}

The second composition morphism $\circ_{1}\circ(f_{a},f_{b})$ is given by the following fibered
product. Let $\beta',\beta''$ such that $\beta'+\beta''=\beta$.

\resizebox{1\linewidth}{!}{
  \begin{minipage}{\linewidth}\begin{align}
  \label{eq:10}
  \xymatrix{&&\coprod_{\beta}\mathbb{R}\overline{\mathcal{M}}^{\sigma}_{0,a+1}(X,\beta)
               \times_{X}\coprod_{\beta}\mathbb{R}\overline{\mathcal{M}}^{\tau}_{0,b+1}(X,\beta) \ar[rd]\ar[ld]&&
\\&\coprod_{\beta}\mathbb{R}\overline{\mathcal{M}}^{\sigma}_{0,a+1}(X,\beta)\times X^{b}
              \ar[rd]^{e_{a+1},\Id_{X^{b}}} \ar[ld]_{e_{1}, \ldots ,e_{a},\Id_{X^{a}}} &&\coprod_{\beta}\mathbb{R}\overline{\mathcal{M}}^{\tau}_{0,a+1}(X,\beta)
              \ar[rd]^{e_{b+1}} \ar[ld]_{e_{1}, \ldots ,e_{b}} \\X^{a}\times X^{b} && X \times X^{b} && X}
\end{align}
\end{minipage}
}
Let fix $\beta$.
Finally, the $2$-morphism $\alpha$ is given by the gluing morphism
\begin{align}
  \label{eq:12}
  \alpha :\coprod_{\stackrel{\beta',\beta''}{\beta'+\beta''=\beta}}\mathbb{R}\overline{\mathcal{M}}^{\sigma}_{0,a+1}(X,\beta')
               \times_{X}\mathbb{R}\overline{\mathcal{M}}^{\tau}_{0,b+1}(X,\beta'') \to\mathbb{R}\overline{\mathcal{M}}^{\sigma\circ\tau}_{0,a+b+1}(X,\beta)
\end{align}
Notice that we can glue the stable maps denoted by $(C,x_{1}, \ldots ,x_{a+1},f)$ and
$(\widetilde{C},\widetilde{x}_{1}, \ldots ,\widetilde{x}_{b+1},\widetilde{f})$ because  the fiber
product is over $X$ which means that $f(x_{a+1})=\widetilde{f}(\widetilde{x}_{1})$. 
This morphism $\alpha$ is surjective but not injective on points. To see the non injectivity,
consider Figure \ref{fig:geo}, then the
gluing curves are the same. Notice that  by stability condition, we have $\beta_{2}\neq 0$. The two
couple of curves $(C_{1}\circ C_{2}, C_{3})$ and $(C_{1},C_{2}\circ C_{3})$ are in two different
connected components of
\begin{displaymath}
  \coprod_{\stackrel{\beta',\beta''}{\beta'+\beta''=\beta}}\mathbb{R}\overline{\mathcal{M}}^{\sigma}_{0,a+1}(X,\beta')
               \times_{X}\mathbb{R}\overline{\mathcal{M}}^{\tau}_{0,b+1}(X,\beta'').
\end{displaymath}

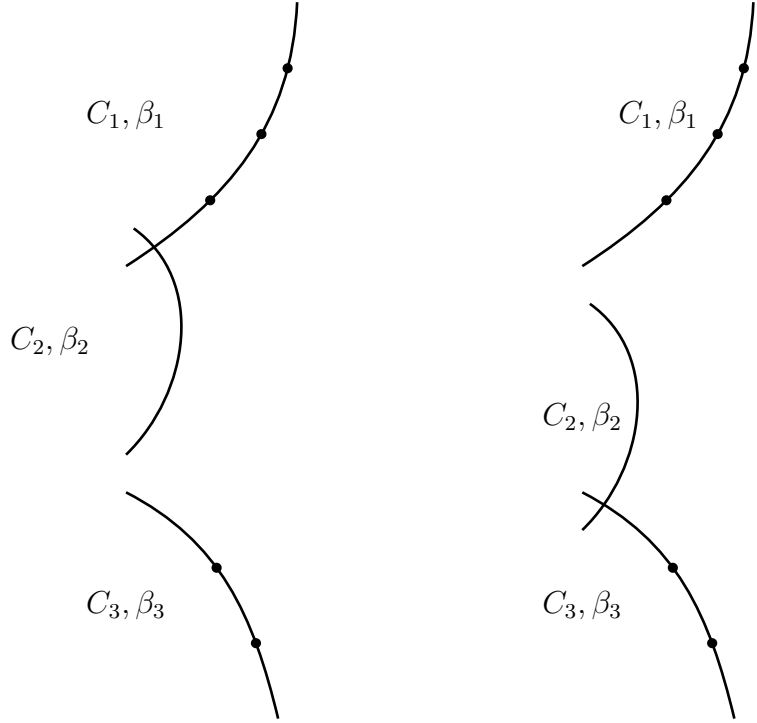
\begin{figure}
  \begin{tikzpicture}
         \ThreeCurveVex{0}{0}{2.25}{3.5}{firstCurve}
\draw (0,2) node   {$C_{1},\beta_{1}$};
  \ZeroCurveVex{0}{-2.5}{.1}{3}{secondCurve}
\draw (-1,-1) node   {$C_{2},\beta_{2}$};
\TwoCurveCave{0}{-3}{2}{-3}{thirdCurve}
\draw (0,-4.5) node   {$C_{3},\beta_{3}$};

         \ThreeCurveVex{6}{0}{2.25}{3.5}{firstCurve}
    \ZeroCurveVex{6}{-3.5}{.1}{3}{secondCurve}
\TwoCurveCave{6}{-3}{2}{-3}{thirdCurve}
\draw (7,2) node   {$C_{1},{\beta}_{1}$};
\draw (6,-2) node   {$C_{2},{\beta}_{2}$};
\draw (6,-4.5) node   {$C_{3},{\beta}_{3}$};

     \label{fig:geo}
   \end{tikzpicture}
  \caption{Geometric reason of the lax action} 
\end{figure}




\section{Proof of our main result}
\label{sec:proof-our-main}
\subsection{Brane action}
In this section, we explain the main theorem of  \cite{Toen-operation-branes-2013}. This theorem has
a lot of prerequisites (like $\infty$-operads, unital and  coherent operads) that are too
complicated for this survey. We refer to the definition of $\infty$-operads by Lurie \cite[Definition
2.1.1.8]{Lurie-higher-algebra} and to the  Definition 3.3.1.4 for the notion of coherent $\infty$-operad. 


\begin{thm}[see Theorem \cite{Toen-operation-branes-2013}]\label{thm:brane}
  Let $\mathcal{O}^\otimes$ be an $\infty$-operad in the $\infty$-category of spaces such that
  \begin{enumerate}
  \item $\mathcal{O}^\otimes(0)=\mathcal{O}^\otimes(1)$ are contractible.
  \item the operad is unital and coherent
  \end{enumerate}
Then $\mathcal{O}(2)$ is a $\mathcal{O}^\otimes$-algebra in the $\infty$-category of
co-correspondence.
\end{thm}

\begin{example}
  We will illustrate the hypothesis and the conclusion of this theorem for the operad
  $\mathcal{O}(n):=\overline{\mathcal{M}}_{0,n+1}$. We choose this example because it is a
  well-known operad and it is easier to explain. Notice that to prove (see \S \ref{thm:brane,Costello}) our main theorem, we need to
  apply to an other operad which is $\coprod_{\beta} \mathfrak{M}_{0,n+1,\beta}$ but the main ideas are
  the same.
Notice that  we set
$\overline{\mathcal{M}}_{0,1}=\overline{\mathcal{M}}_{0,2}:=\pt$ (with the usual definition they are
empty). By definition, we impose that
$\mathcal{O}(1)$ is the unit. 
For the operad $\mathcal{O}$, the following diagram is  cartesian (See below for an explanation).
\begin{displaymath}
  \xymatrix{\mathcal{O}(n)\times \mathcal{O}(m+1) \coprod_{\mathcal{O}(2)\times
      \mathcal{O}(n)\times\mathcal{O}(m)}\mathcal{O}(n+1)\times\mathcal{O}(m) \ar[r]
    \ar[d]^{q}&\mathcal{O}(n+m) \ar[d]^{p}\\\mathcal{O}(n)\times\mathcal{O}(m) \ar[r]^{\circ}& \mathcal{O}(n+m-1)}
\end{displaymath}
This property was called  of ``configuration type'' in \cite{Toen-operation-branes-2013}. Notice
that in the context of \cite[Definition 3.3.1.4]{Lurie-higher-algebra}, this notion was called ``coherent''.
As $p$ is flat, we need to prove that it is a cartesian diagram in the stack category.  Let
$(C_{1},x_{1}, \ldots ,x_{n+1})$ be in $\mathcal{O}(n)$
and $(C_{2},y_{1}, \ldots ,y_{m+1})$ be in $\mathcal{O}(m)$.
As
$\mathcal{O}(n+1)\to\mathcal{O}(n)$ is the universal curve, we deduce that
$q^{-1}(C_{1},C_{2})=C_{1}\coprod_{\pt}C_{2}$ which is exactly $C_{1}\circ C_{2}$. This implies that
the diagram above is cartesian.

Let us explain now the conclusion of this theorem. Notice  that $\mathcal{O}(2)=\overline{\mathcal{M}}_{0,3}$ is a point.
The statement means that we have a  morphism  of $\infty$-operad that is a family of morphisms
 \begin{displaymath}
 \varphi_{n}: \mathcal{O}(n)\to \underline{\Hom}^{\cocor}(\coprod_{i=1}^{n}\mathcal{O}(2), \mathcal{O}(2))
\end{displaymath}
where the morphism $(\varphi_{n})$ are compatible with the composition law. 
The $\underline{\Hom}$ is the same meaning that in \S \ref{sec:definition,End}. The category of
co-correspondances is in the same spirit as correspondance (See \S \ref{sec:categ,corr}) but with
the arrows in the other directions. 
The morphism $\varphi_{n}$ is given by the following diagram
\begin{align}\label{eq:16}
\xymatrix{  \mathcal{O}(n)\times \coprod_{i=1}^{n} \mathcal{O}(2)\ar[r]^-{\circ} \ar[rd]& \mathcal{O}(n+1) \ar[d] & \ar[l]_-{\circ'}\mathcal{O}(2)\times
  \mathcal{O}(n) \ar[ld]   \\ & \mathcal{O}(n)&}
\end{align}
Let explain this diagram with $\mathcal{O}(n)=\overline{\mathcal{M}}_{0,n+1}$. We have
\begin{enumerate}
\item The morphism $\mathcal{O}(n+1)\to \mathcal{O}(n)$ is to forget the last marked point. 
\item The map $\circ:\mathcal{O}(n)\times \coprod_{i=1}^{n}\mathcal{O}(2) \to \mathcal{O}(n+1)$ is
  given by
the $n$ possible gluings of the third marked point of $\mathcal{O}(2)=\overline{\mathcal{M}}_{0,3}$
with one of the marked points $x_{i}$ for $i\in \{1, \ldots ,n\}$ in $\mathcal{O}(n)$.
\item The $\circ'$ is the
gluing of last marked point $x_{n+1}$ of $\mathcal{O}(n)$ with the third of $\mathcal{O}(2)$.
\end{enumerate}

\end{example}

\subsection{Sketch of proof of Theorem \ref{thm,main}}
In this section, we explain how to apply Theorem \ref{thm:brane} to get our main theorem.

Here we take $\mathcal{O}(n)=\coprod_{\beta}\mathfrak{M}_{0,n+1,\beta}$.   This is an operad in
algebraic stack. One can check that all we
said before in the previous section for $\overline{\mathcal{M}}_{0,n+1}$ works as well for
$\coprod_{\beta}\mathfrak{M}_{0,n+1,\beta}$. 

 Let $X$ be a smooth projective variety.
We apply the functor $\mathbb{R}\Hom_{/\mathfrak{M}_{0,n+1,\beta}}(-,X\times
\mathfrak{M}_{0,n+1,\beta})$ to Theorem \ref{thm:brane}. As the source curve of a stable map  may
not be a stable curve, we need to use Theorem \ref{thm:brane} with an other operad than
$\overline{\mathcal{M}}_{0,n+1}$. That's why we use $\coprod_{\beta}\mathfrak{M}_{0,n+1,\beta}$.
We deduce the following result. 
\begin{thm}\label{thm:brane,Costello}
The variety $X$ is an $\mathfrak{M}^{\otimes}$-algebra in the category of correspondances in derived
stacks. The algebra structure is given by the
\begin{displaymath}
  \xymatrix{&\mathbb{R}\overline{\mathcal{M}}_{0,n+1}(X,\beta)\ar[rd] \ar[dl] & \\ X^{n}\times \mathfrak{M}_{0,n+1,\beta} && X \times \mathfrak{M}_{0,n+1,\beta}}
\end{displaymath}
\end{thm}

\begin{remark}
To apply Theorem \ref{thm:brane}, we need to do several modifications
  \begin{enumerate}
  \item Notice that in this statement, the action is strong that means that the lax morphisms are
    equivalences (See \S \ref{sec:lax,morphism}). The
geometrical reason is the following. We can repeat the construction of \S \ref{sec:lax,morphism}
replacing $\overline{\mathcal{M}}_{0,n+1}$ by $\mathfrak{M}_{0,n,\beta}$. The difference is that the
forgetting morphism  $q:\overline{\mathcal{M}}_{0,n+1}(X,\beta)\to \overline{\mathfrak{M}}_{0,n+1,\beta}$
 does not contract any component of the curve. 
More precisely, let  $\sigma \in
\mathfrak{M}_{0,a+1,\beta}$ and $\tau \in \mathfrak{M}_{0,b+1,\beta'}$. Denote by
\begin{displaymath}
\mathbb{R}\overline{\mathcal{M}}^{\sigma}_{0,a+1}(X,\beta')=q^{-1}(\sigma).  
\end{displaymath}
 Take care that in \S
\ref{sec:lax,morphism}, we use
$\mathbb{R}\overline{\mathcal{M}}^{\sigma}_{0,a+1}(X,\beta')=p^{-1}(\sigma)$ where $p:\overline{\mathcal{M}}_{0,n+1}(X,\beta)\to \overline{\mathcal{M}}_{0,n+1}$.
Writing the same kind of
diagram as \eqref{eq:10} we get the corresponding $\alpha$ given by
  \begin{align}
  \label{eq:31}
  \widetilde{\alpha} :\mathbb{R}\overline{\mathcal{M}}^{\sigma}_{0,a+1}(X,\beta')
               \times_{X}\mathbb{R}\overline{\mathcal{M}}^{\tau}_{0,b+1}(X,\beta'') \to\mathbb{R}\overline{\mathcal{M}}^{\sigma\circ\tau}_{0,a+b+1}(X,\beta)
\end{align}
which is now an isomorphism because from the glued curve, there is a unique possibility to cut it
with respect to $\sigma$ and $\tau$.

  \item First, Theorem \ref{thm:brane} apply only to operads in spaces and here we have operads in
    derived stacks. This can be done using non-planar rooted trees and dendroidal sets. More
    precisely, one can enrich $\infty$-operads using Segal functor from the nerve of $\Omega^{op}$
    to derived stacks. Thanks to the work of \cite{1606.03826} and \cite{1305.3658} these two
    definitions coincide on topological spaces.
  \item Second, the condition $\mathcal{O}(0)=\mathcal{O}(1)=\pt$ is not satisfied by
    $\mathfrak{M}_{0,n,\beta}$. So we impose that for any $\beta\neq 0$,
    $\mathfrak{M}^{\fake}_{0,1,\beta}=\mathfrak{M}^{\fake}_{0,2,\beta}=\emptyset$  and that
    $\mathfrak{M}^{\fake}_{0,1,0}=\mathfrak{M}^{\fake}_{0,2,0}=\pt$ is with
    $\mathfrak{M}^{\fake}_{0,2,0}$ being the neutral element.
  \item An other issue is that $\mathfrak{M}_{0,n,\beta}$ is not a coherent operad because the
    inclusion of schemes in derived stacks does not commute with pushouts even along closed
    immersion. We only have a canonical morphism
    \begin{displaymath}
      \theta: C_{1}\coprod^{dst}_{\pt}C_{2} \to
    C_{1}\coprod^{sch}_{\pt}C_{2}
    \end{displaymath}
 Nevertheless, most of the proof of Theorem \ref{thm:brane} is
    still valid and we know that the functor $\mathbb{R}\Hom (-,X)$ will see $\theta$ as an equivalence.
  \end{enumerate}
\end{remark}

The next step in order to prove Theorem \ref{thm,main} is to understand the morphism of operads
\begin{displaymath}
  \coprod_{\beta}\mathfrak{M}_{0,n+1,\beta} \to \overline{\mathcal{M}}_{0,n+1}.
\end{displaymath}
Embedding this morphism in the $\infty$-operads, it turns out that this morphism is a lax
morphism of operads. This is the reason why the final action in Theorem \ref{thm,main} is lax.


\section{Comparison with other definition}

\subsection{Quantum product in cohomology and in $G_{0}$-theory}

In this section, we review the definition of the quantum product in cohomology and in $G_{0}$-theory.
Recall that $X$ is a smooth projective variety. Givental-Lee defined in \cite{MR2040281} the
Gromov-Witten invariants in $G_{0}$-theory. For that they defined a virtual structure sheaf, denoted by $\mathcal{O}^{\vir}_{\overline{\mathcal{M}}_{g,n}(X,\beta)}$, on the moduli
space of stable maps. Recall the morphism $e_{i}:\overline{\mathcal{M}}_{g,n}(X,\beta)\to X$ are the
evaluation morphism at the $i$-th marked point.
For any $E_{1}, \ldots ,E_{n}\in G_0(X)$, the Gromov-Witten invariants in $G_{0}$-theory are
\begin{displaymath}
  \langle E_{1}, \ldots ,E_{n} \rangle^{G_{0}}_{0,n,\beta}:=\chi\left(\bigotimes_{i=1}^{n} e_{i}^{*}E_{i}
  \otimes\mathcal{O}^{\vir}_{\overline{\mathcal{M}}_{0,n}(X,\beta)}  \right) \in \mathbb{Z}
\end{displaymath}
where $\chi(.)$ is the Euler characteristic.

Let $\NE(X)$ be  the Neron-Severi group of $X$ that is the subset of $H_{2}(X,\mathbb{Z})$ generated
by image of curves in $X$.
\begin{defn}
Let $\gamma_{1},\gamma_{2} \in H^{*}(X)$. The quantum product in $H^{*}(X)$ is defined by
\begin{align}\label{eq:21}
  \gamma_{1}\bullet^{H^{*}} \gamma_{2}=\sum_{\beta \in \NE(X)} Q^{\beta}{\ev_{3}}_{*}\left(\ev_{1}^{*}\gamma_{1}\cup
    \ev_{2}^{*}\gamma_{2}\cap [\overline{\mathcal{M}}_{0,3}(X,\beta)]^{\vir}\right). 
\end{align}
\end{defn}
One can see this product as a formal power series in $Q$. Hence, the quantum product lies in
$H^{*}(X)\otimes \Lambda$ where $\Lambda$ is the Novikov ring i.e., it is the algebra generated by
$Q^{\beta}$ for $\beta \in \NE(X)$.

We will recall the definition of the virtual class $\left[\overline{\mathcal{M}}_{0,n}(X,\beta)\right]^{\vir}$ (defined by Behrend-Fantechi) and the virtual
 sheaf $\mathcal{O}^{\vir}_{\overline{\mathcal{M}}_{g,n}(X,\beta)}$ (defined by Lee \cite{MR2040281})   in \S \ref{sec:virtual-object-from-1} and \S \ref{sec:virtual-object-from}.

In $G_{0}$-theory, we define the quantum product with the following formula.
\begin{defn}\label{eq:22}
  Let $F_{1},F_{2} \in G_{0}(X)$. The quantum product in $G_{0}$-theory is defined to be the
  element in $G_{0}(X)\otimes \Lambda$
 \begin{displaymath}
 \resizebox{1\linewidth}{!}{
   \begin{minipage}{\linewidth}
 \begin{align*}
  F_{1}\bullet^{G_{0}} F_{2} =\sum_{\beta \in \NE(X)} Q^{\beta}{\ev_{3}}_{*}\left(\ev_{1}^{*}F_{1}\otimes
     \ev_{2}^{*}F_{2}\otimes \sum_{r\in\mathbb{N}}\sum_{\stackrel{
  (\beta_{0}, \ldots ,\beta_{r})\mid}{\sum\beta_{i}=\beta}}
 (-1)^{r}\mathcal{O}^{\vir}_{\overline{\mathcal{M}}_{0,3}(X,\beta_{0})}\otimes
 \mathcal{O}^{\vir}_{\overline{\mathcal{M}}_{0,2}(X,\beta_{1})} \cdots \otimes
 \mathcal{O}^{\vir}_{\overline{\mathcal{M}}_{0,2}(X,\beta_{r})}\right) 
 \end{align*}
 \end{minipage}}
\end{displaymath}
\end{defn}

The term $r=0$ in the formula in Definition \ref{eq:22} is of the same shape
\eqref{eq:21}. One has to understand the other terms, i.e. $r>0$, are ``corrections
terms''.

\subsection{About the associativity}
\label{sec:about-associativity}

The most important property of these two products is the associativity. 
It is 
proved by Kontsevich-Manin \cite{MR1369420} (See also \cite{MR1492534}) that the quantum product in
cohomology is associative. Notice that the key formula for the associativity is given in Theorem
\ref{thm:gluing,virt,coho} which states that virtual classes behave with respect to the morphisms
$\alpha$'s and the gluing morphisms. Recall that the morphisms $\alpha$'s are the one that appear in the lax action \eqref{eq:12}. 

Later, when Givental and Lee (See \cite{MR2040281}) try to define a quantum product in $G_{0}$-theory they want an
associative product. If one put the same kind of formula as in \eqref{eq:21}, the product is not
associative. Hence the key observation of Givental and Lee is Theorem \ref{thm:gluing,sheaf,virt}
which is the analogue of Theorem \ref{thm:gluing,virt,coho} in $G_{0}$-theory that is how the
virtual sheaves behave with respect to the morphisms $\alpha$'s and the gluing morphisms.

Our contribution to this question is Theorem \ref{thm:colim} which is the geometric explanation that explains
the two Theorems \ref{thm:gluing,sheaf,virt} and \ref{thm:gluing,sheaf,virt}.

Notice that Givental-Lee packed the complicated formula of \ref{eq:22} in a very clever way. Notice
that $\overline{\mathcal{M}}_{0,2}(X,\beta)=\overline{\mathcal{M}}_{0,2} \times X$ is empty if
$\beta=0$. As before put $\overline{\mathcal{M}}_{0,2}=\pt$. Then we put 
\begin{align}
  \label{eq:23}
\mathcal{O}^{\vir}_{\overline{\mathcal{M}}_{0,2}}:= \mathcal{O}_{X} + \sum_{\stackrel{\beta\in
  \NE(X)}{\beta\neq 0}}
  Q^{\beta}\mathcal{O}^{\vir}_{\overline{\mathcal{M}}_{0,2}(X,\beta)} \in G_{0}(X)\otimes \Lambda
\end{align}
Let invert the Formula above formally in $G_{0}(X)\otimes \Lambda$. The terms in front of
$Q^{\beta}$ is 
\begin{align}\label{eq:24}
  \sum_{r\in\mathbb{N}}\sum_{\stackrel{
  (\beta_{0}, \ldots ,\beta_{r})\mid}{\sum\beta_{i}=\beta}}
 (-1)^{r}\mathcal{O}^{\vir}_{\overline{\mathcal{M}}_{0,2}(X,\beta_{0})}\otimes
 \mathcal{O}^{\vir}_{\overline{\mathcal{M}}_{0,2}(X,\beta_{1})} \cdots \otimes
 \mathcal{O}^{\vir}_{\overline{\mathcal{M}}_{0,2}(X,\beta_{r})}
\end{align}
The Formula \eqref{eq:23} and \eqref{eq:24} are the reason of the ``metric'' (See Formula (16) in \cite{MR2040281}
for more details)
because one can express in a compact form the Formula \eqref{eq:22} using the inverse of the metric.

\subsection{Key diagram}

Let us consider the following homotopical fiber product. Let $n_{1},n_{2}\in\mathbb{N}_{\geq 2}$. Put $n=n_{1}+n_{2}$.
\begin{align}\label{eq:key,diag}
  \xymatrix{ Z_{\beta}\ar[r]\ar[d]& \mathbb{R}\overline{\mathcal{M}}_{0,n}(X,\beta) \ar[d]^{p}\\
  \overline{\mathcal{M}}_{0,n_{1}+1} \times \overline{\mathcal{M}}_{0,n_{2}+1} \ar[r]^-{g}& \overline{\mathcal{M}}_{0,n}}
\end{align}
The fiber over a point $(\sigma,\tau)$ is denoted by
$\overline{\mathcal{M}}^{\sigma\circ\tau}(X,\beta)$ in \S~\ref{sec:lax,morphism} that is stable
maps where the curve stabilise to $\sigma\circ \tau$. In Figure \ref{fig:tree,p1}, we have an
example of a fiber over $\sigma\circ \tau$ where we have a tree of $\mathbb{P}^{1}$ in the middle.

\begin{figure}[ht]
\centering
  \begin{tikzpicture}[scale=0.8]
    \ThreeCurveVex{0}{0}{2.25}{3.5}{firstCurve}
    \ZeroCurveVex{0}{-2.5}{.1}{3}{secondCurve}
  \ZeroCurveVex{0}{-4.5}{.1}{3}{fourthCurve}
\ZeroCurveVex{0}{-6.5}{.1}{3}{fifthCurve}
\TwoCurveCave{0}{-6}{2}{-3}{thirdCurve}

    \ThreeCurveVex{5}{-3}{2.25}{3.5}{firstCurve}
\TwoCurveCave{5}{-2.5}{2}{-3}{thirdCurve}
\draw (-1,2) node {$\sigma,\beta_{0}$};
\draw (-1,-1) node {$C_{1},\beta_{1}$};
\draw (-1,-3) node {$C_{2},\beta_{2}$};
\draw (-1,-5) node {$C_{3},\beta_{3}$};
\draw (-1,-7.5) node {$\tau,\beta_{4}$};

\draw [>=latex,->,black] (2,-2.7)-- (4,-2.7) node [above left, midway] {$p$};
\draw (1.7,2.6) node {$x_{1}$}; 
\draw (1.3,1.8) node {$x_{2}$}; 
\draw (0.7,1) node {$x_{3}$}; 
\draw (0.7,-6.9) node {$x_{4}$}; 
\draw (1.2,-7.8) node {$x_{5}$}; 

\draw (6.7,-.2) node {$x_{1}$}; 
\draw (6.5,-1) node {$x_{2}$}; 
\draw (6,-1.7) node {$x_{3}$}; 
\draw (5.7,-3.6) node {$x_{4}$}; 
\draw (6.3,-4.5) node {$x_{5}$}; 

\draw (7.9,-2.7) node {$\sigma\circ\tau \in \overline{\mathcal{M}}_{0,5}$};

  \end{tikzpicture}
\caption{Example of a stable map above $\sigma\circ\tau$ with a tree of $\mathbb{P}^{1}$ in the
  middle. The tree $C_{1}\circ C_{2} \circ C_{3}$ is contracting by $p$ to the node of $\sigma\circ \tau$.} \label{fig:tree,p1}
\end{figure}
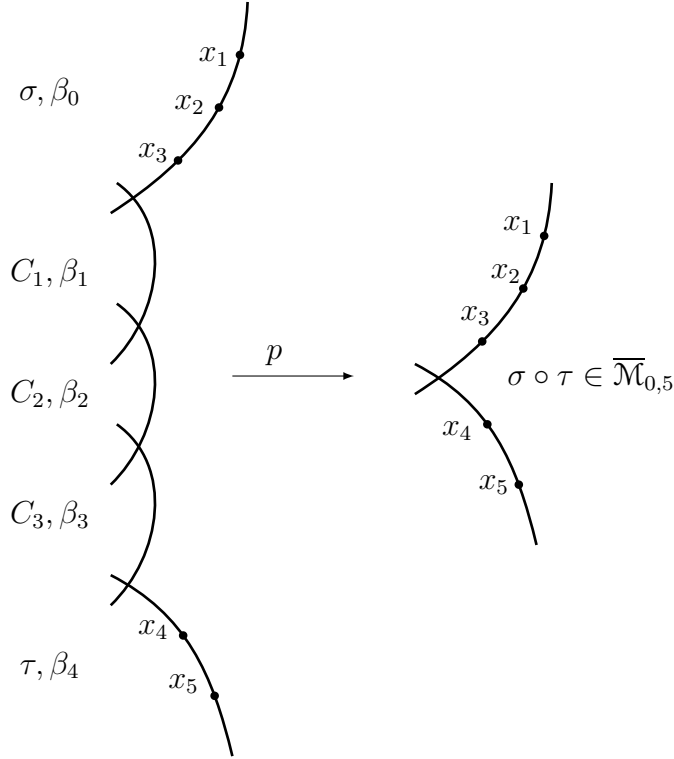

Using the universal property of the fiber product we get the morphism (see \eqref{eq:12})
\begin{align}\label{eq:alpha,asso}
  \alpha
  :\coprod_{\beta'+\beta''=\beta}\mathbb{R}\overline{\mathcal{M}}_{0,n_{1}+1}(X,\beta')\times_{X}\mathbb{R}\overline{\mathcal{M}}_{0,n_{2}+1}(X,\beta'')
  \to Z_{\beta}
\end{align}
where the left hand side is defined by the following homotopical fiber product
\begin{align}
  \label{eq:13}
\xymatrix{\mathbb{R}\overline{\mathcal{M}}_{0,n_{1}+1}(X,\beta')\times_{X}\mathbb{R}\overline{\mathcal{M}}_{0,n_{2}+1}(X,\beta'')
  \ar[r] \ar[d]
  &\mathbb{R}\overline{\mathcal{M}}_{0,n_{1}+1}(X,\beta')\times\mathbb{R}\overline{\mathcal{M}}_{0,n_{2}+1}(X,\beta'') \ar[d]^-{e_{1},e_{n_{2}+1}}\\
  X \ar[r]^-{\Delta} & X\times X} 
\end{align}

The heart of the associativity of the quantum products in cohomology (see Theorem
\ref{thm:gluing,sheaf,virt} for $G_{0}$-theory) is the following statement.
\begin{thm}[Theorem 5.2 \cite{MR1467172}]\label{thm:gluing,virt,coho}
 We have the following equality in the Chow ring of the truncation of $Z_{\beta}$.
    \begin{align}
      \label{eq:14}
      \alpha_{*}\left(\sum_{\beta'+\beta''=\beta} \Delta^{!}\left([\overline{\mathcal{M}}_{0,n_{1}+1}(X,\beta')]^{\vir}\otimes[\overline{\mathcal{M}}_{0,n_{2}+1}(X,\beta'')]^{\vir})\right)\right)=g^{!}[\overline{\mathcal{M}}_{0,n}(X,\beta)]^{\vir}
    \end{align}
\end{thm}

\begin{remark}\label{rem,orientation,coho}
  In \cite{MR1431140}, Behrend proves that the virtual class satisfies five properties, called
  \textit{orientation} (see \S 7 in \cite{MR1412436}), namely: mapping to a point, products, cutting
  edges, forgetting tails and isogenies. The formula \eqref{eq:14} is a combination of cutting tails
  and isogenies.
\end{remark}

The analogue statement in $G_{0}$-theory need a bit more of notations. We denote
\begin{displaymath}
  \mathbb{R}\overline{\mathcal{M}}_{g,n}(X,\beta):=\mathbb{R}X_{g,n,\beta}.
\end{displaymath}
  Let $r,n_{1},n_{2}$ be
in $\mathbb{N}$ with $n_{1}+n_{2}=n$ and let $\beta$ be in $\NE(X)$. Let
$\underline{\beta}=(\beta_{0}, \ldots ,\beta_{r})$ be a partition of $\beta$. Notice that there is
only a finite number of partition.


We denote by 
\begin{displaymath}
\resizebox{1\linewidth}{!}{
  \begin{minipage}{\linewidth}
\begin{align*}
  \mathbb{R}{X}_{0,n_{1},n_{2},\underline{\beta}}:=\mathbb{R}X_{0,n_{1}+1,\beta_{0}}\times_{X}\mathbb{R}X_{0,2,\beta_{1}}\times_{X}\cdots\times_{X}\mathbb{R}X_{0,2,\beta_{r-1}}\times_{X}\mathbb{R}X_{0,n_{2}+1,\beta_{r}}
\end{align*}
\end{minipage}}
\end{displaymath}

We generalize the situation of \eqref{eq:14} by the following homotopical cartesian diagram
\begin{align}
  \label{eq:13,diag,r}
\xymatrix{ \mathbb{R}X_{0,n_{1},n_{2},\underline{\beta}}
  \ar[r] \ar[d]
  &\mathbb{R}X_{0,n_{1}+1,\beta_{0}}\times\left(\prod_{k=1}^{r-1}\mathbb{R}X_{0,2,\beta_{i}}\right)\times\mathbb{R}X_{0,n_{2}+1,\beta_{r}} \ar[d]\\
  X^{r} \ar[r]^-{\Delta^{r}} & (X\times X)^{r}} 
\end{align}

Gluing all
the stable maps and using the universal property of $Z_{\beta}$, we have a morphism
\begin{align}\label{eq:alpha,r}
  \alpha_{r}: \coprod_{\beta=\sum_{i=0}^{r}\beta_{i}}
  \mathbb{R}X_{0,n_{1},n_{2},\underline{\beta}} \to Z_{\beta}
\end{align}
Notice that $\alpha_{1}$ is the $\alpha$ of  \eqref{eq:12}

Finally, we can state the analogue of Theorem \ref{thm:gluing,virt,coho} in $G_{0}$-theory.
\begin{thm}[Proposition 11 in \cite{MR2040281}]\label{thm:gluing,sheaf,virt}
  We have the following equality in the $G_{0}$-group of the truncation of $Z_{\beta}$.

\resizebox{1\linewidth}{!}{
  \begin{minipage}{\linewidth}   \begin{align*}
      \label{eq:15}
     \sum_{r\in\mathbb{N}} (-1)^{r}{\alpha_{r}}_{*}\left(\sum_{\sum_{i=0}^{r}\beta_{i}=\beta}
      (\Delta^{r})^{!}\left(\mathcal{O}^{\vir}_{X_{0,n_{1}+1\beta_{0}}}\otimes\mathcal{O}^{\vir}_{X_{0,2,\beta_{1}}}\otimes\cdots
      \otimes \mathcal{O}^{\vir}_{X_{0,2,\beta_{r-1}}}\otimes
      \mathcal{O}^{\vir}_{X_{0,n_{2}+1,\beta_{r}}}\right)\right)=g^{!}\mathcal{O}^{\vir}_{X_{0,n,\beta}}
    \end{align*}\end{minipage}}
\end{thm}

\begin{remark}\label{rem:orientation,Ktheory}
  \begin{enumerate}
  \item   Comparing Theorem \ref{thm:gluing,virt,coho} with Theorem \ref{thm:gluing,sheaf,virt}, we see that
  the formulas are more complicated in $G_{0}$-theory. We see that moduli spaces
  of the kind $\overline{\mathcal{M}}_{0,2}(X,\beta)$ appears in $G_{0}$-theory. This corresponds to stable curve
  with tree of $\mathbb{P}^{1}$ in the middle (see Figure \ref{fig:tree,p1}). Notice that this is
  the same reason why the action of the main Theorem \ref{thm,main} is lax.
\item Also in $G_{0}$-theory, there are 5 axioms, called orientation (see Remark \ref{rem,orientation,coho}), for the virtual sheaf
  $\mathcal{O}^{\vir}_{\overline{\mathcal{M}}_{g,n}(X,\beta)}$. They are proved by Lee in \cite{MR2040281}.
  \end{enumerate}
\end{remark}

Denote by
\begin{displaymath}
X_{r,\beta}:=  \coprod_{\sum\beta_{i}=\beta}\mathbb{R}X_{0,n_{1}+1,\beta_{0}}\times_{X}\mathbb{R}X_{0,2,\beta_{1}}\times_{X}\cdots
\times_{X} \mathbb{R}X_{0,2,\beta_{r-1}}\times_{X}\mathbb{R}X_{0,n_{2}+1,\beta_{r}}
\end{displaymath}
We deduce a semi-simplicial object in the category of derived stacks
where the $r+1$-morphisms from $X_{r+1,\beta}\to X_{r,\beta}$ are given by gluing two stable maps together.
We have 

 \begin{tikzpicture}
\matrix (m) [matrix of math nodes, row sep=3em, column sep=3.2pc,
  text width=2pc, text height=1pc, text depth=.5pc] { 
    X_{0,\beta} & X_{1,\beta} & X_{2,\beta} & \cdots \\
  }; 

\path[<-] 
(m-1-1.15) edge node[above] {}  (m-1-2.165)
(m-1-1.-15) edge (m-1-2.-165);
\path[<-]
(m-1-2.28) edge node[above] {} (m-1-3.152)
(m-1-2) edge (m-1-3)
(m-1-2.-28) edge (m-1-3.-152);
\path[<-]
(m-1-3.37) edge node (t) {} node[above] {} (m-1-4.143)
(m-1-3.-37) edge node (b) {} (m-1-4.-143);


\path[dotted]
(t) edge (b);
\end{tikzpicture}

 Moreover, for any $r$ we have a morphism of gluing all stable maps from $X_{r,\beta}\to Z_{\beta}$ hence a morphism $\colim
X_{\bullet,\beta} \to Z_{\beta}$.

The following theorem was not proved in \cite{2015arXiv150502964M}. We will prove it in the appendix.
\begin{thm}\label{thm:colim}
  We have that $\colim
X_{\bullet,\beta} =Z_{\beta}$.
\end{thm}

\subsection{Virtual object from derived algebraic geometry}
\label{sec:virtual-object-from-1}
In this section, we explain how derived algebraic geometry will provide a sheaf in
$G_{0}(\overline{\mathcal{M}}_{g,n}(X,\beta))$ that we will compare to the virtual sheaf of Lee.

\begin{lem}[See for example \cite{MR3285853} p.192-193]\label{lem:K,iso} Let $X$ be a derived algebraic stack. Denote by $t_{0}(X)$ its truncation.
Denote by  $\iota:t_{0}(X)\hookrightarrow X$ be the closed embedding.
The morphism $\iota_{*}:G_{0}(t_{0}(X)) \to G_0(X)$ is an isomorphism. Moreover we have that
\begin{displaymath}
  (\iota_{*})^{-1}[\mathcal{F}]=  \sum_{i}(-1)^{i}[\pi_{i}(\mathcal{F})]
\end{displaymath}
\end{lem}

Applying this lemma to the situation where $X=\mathbb{R}\overline{\mathcal{M}}_{g,n}(X,\beta)$, we
put
\begin{displaymath}
  \left[\mathcal{O}_{\overline{\mathcal{M}}_{g,n}(X,\beta)}^{\vir,\DAG}\right]:=\iota_{*}^{-1}[\mathcal{O}_{\mathbb{R}\overline{\mathcal{M}}_{g,n}(X,\beta)}].
\end{displaymath}
where the DAG means Derived Algebraic Geometry. Notice that the sheaf
$\mathcal{O}_{\overline{\mathcal{M}}_{g,n}(X,\beta)}^{\vir,\DAG}$ depends on the derived structure
that we put on the moduli space of stable maps.

The following theorem was not stated in \cite{2015arXiv150502964M}.

\begin{thm}\label{thm:orientation}
  The DAG-virtual sheaf
  $\mathcal{O}_{\overline{\mathcal{M}}_{g,n}(X,\beta)}^{\vir,\DAG}$ satisfies the orientation
axiom in $G_{0}$-theory. That is
\begin{enumerate}
\item Mapping to a point. Let $\beta=0$, we have
  \begin{displaymath}
    \mathcal{O}_{\overline{\mathcal{M}}_{g,n}(X,0)}^{\vir,\DAG}=\sum_{i}(-1)^{i}\wedge^{i}(R^{1}\pi_{*}\mathcal{O}_{\mathcal{C}}\boxtimes
    T_{X})^{\vee}
  \end{displaymath}
where $\mathcal{C}$ is the universal curve of $\overline{\mathcal{M}}_{g,n}$ and $\pi:\mathcal{C}\to
\overline{\mathcal{M}}_{g,n}$.
\item Product. We have
  \begin{displaymath}
\mathcal{O}^{\vir,\DAG}_{\overline{\mathcal{M}}_{g_{1},n_{1}}(X,\beta_{1})\times \overline{\mathcal{M}}_{g_{2},n_{2}}(X,\beta_{2})}=\mathcal{O}^{\vir,\DAG}_{\overline{\mathcal{M}}_{g_{1},n_{1}}(X,\beta_{1})}\boxtimes\mathcal{O}^{\vir,\DAG}_{\overline{\mathcal{M}}_{g_{2},n_{2}}(X,\beta_{2})} 
  \end{displaymath}
\item Cutting edges. With the notation of Diagram \eqref{eq:13}, we have
\begin{displaymath}
  \mathcal{O}^{\vir,\DAG}_{\overline{\mathcal{M}}_{g_{1},n_{1}}(X,\beta_{1})\times_{X}\overline{\mathcal{M}}_{g_{2},n_{2}}(X,\beta_{2})}
    =\Delta^{!}\mathcal{O}^{\vir,\DAG}_{\overline{\mathcal{M}}_{g_{1},n_{1}}(X,\beta_{1})\times \overline{\mathcal{M}}_{g_{2},n_{2}}(X,\beta_{2})}
\end{displaymath}
\item Forgetting tails. Forgetting the last marked point marked points, we get a morphism 
  $\pi:\overline{\mathcal{M}}_{g,n+1}(X,\beta)\to \overline{\mathcal{M}}_{g,n}(X,\beta)$. We have
  the following equality.
\begin{displaymath}\pi^{*}\mathcal{O}^{\vir,\DAG}_{\overline{\mathcal{M}}_{g,n}(X,\beta)}= \mathcal{O}^{\vir,\DAG}_{\overline{\mathcal{M}}_{g,n+1}(X,\beta)}.
  \end{displaymath}
\item Isogenies. The are two formulas. The morphism $\pi$ above induces a morphism
  $\psi: \overline{\mathcal{M}}_{g,n+1}(X,\beta)\to \overline{\mathcal{M}}_{g,n+1}
  \times_{\overline{\mathcal{M}}_{g,n}}\overline{\mathcal{M}}_{g,n}(X,\beta)$. With notation of
  Diagram \eqref{eq:key,diag}, we have
\begin{displaymath}
    \psi_{*}\mathcal{O}^{\vir,\DAG}_{\overline{\mathcal{M}}_{g,n+1}(X,\beta)}=g^{!}\mathcal{O}^{\vir,\DAG}_{\overline{\mathcal{M}}_{g,n}(X,\beta)}.
  \end{displaymath}
The second formula is
\begin{displaymath}
   \sum_{r\in\mathbb{N}} (-1)^{r}{\alpha_{r}}_{*}\sum_{\sum_{i=0}^{r}\beta_{i}=\beta}
      \mathcal{O}^{\vir,\DAG}_{X_{0,n_{1}+1}(X,\beta_{0})\times_{X}X_{0,2}(X,\beta_{1})\times_{X}\cdots
      \times_{X}X_{0,2}(X,\beta_{r-1})\times_{X}X_{0,n_{2}+1}(X,\beta_{r})}=g^{!}\mathcal{O}^{\vir,\DAG}_{X_{0,n}(X,\beta)}
\end{displaymath}
where $g$ is defined in the key diagram \eqref{eq:key,diag}.
\end{enumerate}
\end{thm}

Before proving this theorem, we need a  preliminary result.
Consider a homotopical cartesian morphisms of schemes 
\begin{displaymath}
  \xymatrix{X':=X\times_{Y}Y'\ar[rd]^{\iota} \ar@/^2pc/[rrd] \ar@/_2pc/[ddr]&&\\&X\times^{h}_{Y}Y'\ar[r]^-{\widetilde{f}} \ar[d]_-{g}&Y'\ar[d]\\&X \ar[r]^-{f}&Y}
\end{displaymath}
Denote by $X\times_{Y}^{h}Y'$ the homotopical pullback so that we have the closed immersion $\iota :
X' \to X\times_{Y}^{h}Y'$.
Assume that $f$ is a regular closed immersion. We have a rafined Gysin morphism (see
\cite[p.4]{MR2040281}, \cite[ex.18.3.16]{MR1644323} or chapter 6 in \cite{MR801033}) which turns to be 
\begin{align}\label{eq:19}
  f^{!}:G(Y') &\to G(X') \\ \nonumber
[\mathcal{F}_{Y'}]& \mapsto  (\iota_{*})^{-1}\circ\widetilde{f}^{*}
[\mathcal{F}_{Y'}].
\end{align}

\begin{proof}[Proof of Theorem \ref{thm:orientation}]
  (1). Strangely this proof is not easy and we postpone to the Appendix~\ref{sec:proof-theorem}.
(2). This follows from the K\"unneth formula.

(3). We have the following diagram.
\begin{displaymath}
  \xymatrix{X_{g_{1},n_{1},\beta_{1}} \times_{X}X_{g_{2},n_{2},\beta_{2}}\ar[d]_{k}\ar@/^2pc/[rd]^-{h}&
\\
X_{g_{1},n_{1},\beta_{1}} \times^{h}_{X}X_{g_{2},n_{2},\beta_{2}}\ar[d]_-{j}\ar[r]^-{g}&
X_{g_{1},n_{1},\beta_{1}} \times X_{g_{2},n_{2},\beta_{2}}\ar[d]^-{i}
\\
 \mathbb{R}X_{g_{1},n_{1},\beta_{1}}
    \times_{X}\mathbb{R}X_{g_{2},n_{2},\beta_{2}} \ar[d] \ar[r]^-{f}&\mathbb{R}X_{g_{1},n_{1},\beta_{1}}
    \times\mathbb{R}X_{g_{2},n_{2},\beta_{2}} \ar[d]^{e_{i},e_{j}} \\ X \ar[r]^-{\Delta}& X\times X}
\end{displaymath}
We deduce the following equalities
\begin{align*}
  \Delta^{!}\mathcal{O}^{\vir,\DAG}_{X_{g_{1},n_{1},\beta_{1}}\times
    X_{g_{2},n_{2},\beta_{2}}}&=\Delta^{!}(i_{*})^{-1}\mathcal{O}_{\mathbb{R}X_{g_{1},n_{1},\beta_{1}}\times
    \mathbb{R}X_{g_{2},n_{2},\beta_{2}}}\\
&= (k_{*})^{-1}g^{*}(i_{*})^{-1}\mathcal{O}_{\mathbb{R}X_{g_{1},n_{1},\beta_{1}}\times \mathbb{R}X_{g_{2},n_{2},\beta_{2}}}&
 & \mbox{by definition of rafined Gysin morphism} \\
&=(k_{*})^{-1}(j_{*})^{-1}f^{*}\mathcal{O}_{\mathbb{R}X_{g_{1},n_{1},\beta_{1}}\times
  \mathbb{R}X_{g_{2},n_{2},\beta_{2}}} &  &\mbox {by derived base change}\\
&= (k_{*})^{-1}(j_{*})^{-1}\mathcal{O}_{\mathbb{R}X_{g_{1},n_{1},\beta_{1}}\times_{X}
  \mathbb{R}X_{g_{2},n_{2},\beta_{2}}} \\
&= \mathcal{O}^{\vir,\DAG}_{X_{g_{1},n_{1},\beta_{1}}\times_{X}
  X_{g_{2},n_{2},\beta_{2}}}
\end{align*}

(4). As $\widetilde{\pi}:\mathbb{R}\overline{\mathcal{M}}_{g,n+1}(X,\beta)\to
\mathbb{R}\overline{\mathcal{M}}_{g,n}(X,\beta)$ is the universal curve (hence, it is flat) and
$\pi$ is the truncation of $\widetilde{\pi}$. The derived base change formula implies the equality. 

(5). We have the following diagram
\begin{displaymath}
  \xymatrix{\overline{\mathcal{M}}_{g,n+1}(X,\beta)\ar[r]^-{\psi} \ar[d]_-{k}&\overline{\mathcal{M}}_{g,n}\times_{\overline{\mathcal{M}}_{g,n}}\overline{\mathcal{M}}_{g,n}(X,\beta)\ar[r]^-{a}
    \ar[d]^-{j}&\overline{\mathcal{M}}_{g,n}(X,\beta)\ar[d]_-{i}\\
\mathbb{R}\overline{\mathcal{M}}_{g,n+1}(X,\beta)\ar[r]^-{\varphi} \ar@/_2pc/[rd] &\overline{\mathcal{M}}_{g,n+1}\times_{\overline{\mathcal{M}}_{g,n}}\mathbb{R}\overline{\mathcal{M}}_{g,n}(X,\beta)\ar[r]^-{b}\ar[d]&\mathbb{R}\overline{\mathcal{M}}_{g,n}(X,\beta)\ar[d]\\
&\overline{\mathcal{M}}_{g,n+1} \ar[r]^-{c}& \overline{\mathcal{M}}_{g,n}}
\end{displaymath}
Notice that as $c$ is flat, the upper right square is also $h$-cartesian.
We have
\begin{align*}
  c^{!}\mathcal{O}^{\vir,\DAG}_{\overline{\mathcal{M}}_{g,n}(X,\beta)}&=c^{!}(i_{*})^{-1}\mathcal{O}_{\mathbb{R}\overline{\mathcal{M}}_{g,n}(X,\beta)}\\
&=a^{*}(i_{*})^{-1}\mathcal{O}_{\mathbb{R}\overline{\mathcal{M}}_{g,n}(X,\beta)}\\ 
&= (j_{*})^{-1}b^{*}\mathcal{O}_{\mathbb{R}\overline{\mathcal{M}}_{g,n}(X,\beta)}& \mbox{by derived base change}\\
&=(j_{*})^{-1}\mathcal{O}_{\overline{\mathcal{M}}_{g,n+1}\times_{\overline{\mathcal{M}}_{g,n}}\mathbb{R}\overline{\mathcal{M}}_{g,n}(X,\beta)}
\end{align*}
On the other hand, we have
\begin{align*}
  \psi_{*}\mathcal{O}^{\vir,\DAG}_{\overline{\mathcal{M}}_{g,n}(X,\beta)}
  &=\psi_{*}(k_{*})^{-1}\mathcal{O}_{\mathbb{R}\overline{\mathcal{M}}_{g,n+1}(X,\beta)}\\
  &= (j_{*})^{-1}\varphi_{*}\mathcal{O}_{\mathbb{R}\overline{\mathcal{M}}_{g,n+1}(X,\beta)} &
\end{align*}
The formula follows from the equality below which is a consequence of the proof of Proposition 9 in \cite{MR2040281}.
\begin{displaymath}
  \varphi_{*}\mathcal{O}_{\mathbb{R}\overline{\mathcal{M}}_{g,n+1}(X,\beta)}=\mathcal{O}_{\overline{\mathcal{M}}_{g,n}\times_{\overline{\mathcal{M}}_{g,n}}\mathbb{R}\overline{\mathcal{M}}_{g,n}(X,\beta)}
\end{displaymath}

To prove the second formula of (5), we use the key Diagram
\eqref{eq:key,diag}) with Theorem \ref{thm:colim}. Let $g_{1},g_{2},n_{1},n_{2}$ be integers. Put
$g=g_{1}+g_{2}$ and $n=n_{1}+n_{2}$ and denote $\overline{\mathcal{M}}_{i}:=\overline{\mathcal{M}}_{g_{i},n_{i}+1}$.
\begin{displaymath}
  \xymatrix{t_{0}(Z_{\beta})\ar[d]^-{k}\ar@/^{2pc}/[rd]^-{a}\\
\left(\overline{\mathcal{M}}_{1}\times \overline{\mathcal{M}}_{2}\right)\times^{h}\overline{\mathcal{M}}_{g,n}(X,\beta)\ar[r]^-{b}
    \ar[d]^-{j}&\overline{\mathcal{M}}_{g,n}(X,\beta)\ar[d]_-{i}\\
Z_{\beta}\ar[r]^-{c}\ar[d]&\mathbb{R}\overline{\mathcal{M}}_{g,n}(X,\beta)\ar[d]\\
\overline{\mathcal{M}}_{1}\times \overline{\mathcal{M}}_{2} \ar[r]^-{g}& \overline{\mathcal{M}}_{g,n}}
\end{displaymath}

We have
\begin{align*}
  g^{!}\mathcal{O}^{\vir,\DAG}_{\overline{\mathcal{M}}_{g,n}(X,\beta)}&= g^{!}(i_{*})^{-1}\mathcal{O}_{\mathbb{R}\overline{\mathcal{M}}_{g,n}(X,\beta)}\\
&=(k_{*})^{-1}b^{*}(i_{*})^{-1}\mathcal{O}_{\mathbb{R}\overline{\mathcal{M}}_{g,n}(X,\beta)}\\
&=(k_{*}^{-1})(j_{*})^{-1}c^{*} \mathcal{O}_{\mathbb{R}\overline{\mathcal{M}}_{g,n}(X,\beta)}&
                                                                                               \mbox{by derived base change
                                                                                           }\\
&=(j\circ k)_{*} ^{-1}\mathcal{O}_{Z_{\beta}}
\end{align*}
We deduce the formula by observing that $Z_{\beta}$ is the colimit of $X_{\bullet,\beta}$ (see
Theorem \ref{thm:colim}) and that the structure sheaf of a co-limit is the alternating sum of $\mathcal{O}_{X_{r,\beta}}$.

\end{proof}

The last formula of Theorem \ref{thm:orientation} and the third one implies the following corollary.

\begin{cor} We have the following equality in $G_0(t_{0}(Z_{\beta}))$.\\
\resizebox{1\linewidth}{!}{
  \begin{minipage}{\linewidth}   \begin{align*}
      \label{eq:16}
     \sum_{r\in\mathbb{N}} (-1)^{r}{\alpha_{r}}_{*}\left(\sum_{\sum_{i=0}^{r}\beta_{i}=\beta}
      (\Delta^{r})^{!}\left(\mathcal{O}^{\vir,\DAG}_{X_{0,n_{1}+1}(X,\beta_{0})}\otimes\mathcal{O}^{\vir,\DAG}_{X_{0,2}(X,\beta_{1})}\otimes\cdots
      \otimes \mathcal{O}^{\vir,\DAG}_{X_{0,2}(X,\beta_{r-1})}\otimes
      \mathcal{O}^{\vir,\DAG}_{X_{0,n_{2}+1}(X,\beta_{r})}\right)\right)=g^{!}\mathcal{O}^{\vir,\DAG}_{X_{0,n}(X,\beta)}
    \end{align*}\end{minipage}}
\end{cor}

\subsection{Virtual object from perfect obstruction theory}
\label{sec:virtual-object-from}
Here we follow the approach of Behrend-Fantechi \cite{MR1437495} to construct virtual object.

In the following, we denote by $\mathcal{M}$  a Deligne-Mumford stack. The reader can think of
$\mathcal{M}$ being $\overline{\mathcal{M}}_{0,n}(X,\beta)$ as an example.

\begin{defn}
  Let $\mathcal{M}$ be a Deligne-Mumford stack. An element $E^{\bullet}$ in the derived category $D(\mathcal{M})$ in degree
  $(-1,0)$ is a perfect obstruction theory for $\mathcal{M}$ if we have a morphism $\varphi:E^{\bullet}\to
  \mathbb{L}_{\mathcal{M}}$ that satisfies
  \begin{enumerate}
  \item $h^{0}(\varphi)$ is an isomorphism,
  \item $h^{-1}(\varphi)$ is surjective.
  \end{enumerate}
\end{defn}

Let $E^{\bullet}$ be a perfect obstruction theory. Following \cite{MR1437495}, we have the following morphisms.
\begin{enumerate}
\item The morphism $a:C_{\mathcal{M}}\to h^{1}/h^{0}(E^{\vee}_{\bullet})$, where $C_{\mathcal{M}}$ is the intrinsic normal cone and $h^{1}/h^{0}(E_{\bullet}^{\vee})$
is the quotient stack $[E_{-1}^{\vee}/E_{0}^{\vee}]$. To understand how to construct this morphism,
let us simplify the situation. Assume that $\mathcal{M}$ is embedded in something
smooth, i.e $f:\mathcal{M}\hookrightarrow Y$ is a closed embedding with ideal sheaf
$\mathcal{I}$. Then the intrinsic normal cone is the quotient stack $C_{\mathcal{M}}=[C_{\mathcal{M}}Y/f^{*}TY]$ where
$C_{\mathcal{M}}Y:=\Spec \oplus_{n\geq 0} \mathcal{I}^{n}/\mathcal{I}^{n+1}$ is the normal cone of
$f$. In this case, the intrinsic normal sheaf is
$N_{\mathcal{M}}=[N_{\mathcal{M}}Y/f^{*}TY]=h^{1}/h^{0}(\mathbb{L}_{\mathcal{M}}^{\vee})$
where $N_{\mathcal{M}}Y:=\Spec \Sym \mathcal{I}/\mathcal{I}^{2}$. As we have a morphism from the
normal cone to the normal sheaf
$C_{\mathcal{M}}Y \to N_{\mathcal{M}}Y$, we deduce a morphism from the intrinsic  normal cone to
the intrinsic normal sheaf i.e., a morphism
 \begin{equation}\label{eq:17}
 C_{\mathcal{M}}\to N_{\mathcal{M}}
 \end{equation}
Now the morphism of the perfect obstruction theory $\varphi:E^{\bullet}\to  \mathbb{L}_{\mathcal{M}}$ induces a morphism from
\begin{equation}
  \label{eq:18}
    N_{\mathcal{M}} \to [E_{-1}^{\vee}/E_{0}^{\vee}]
  \end{equation}
The morphism $a$ is the composition of the two morphisms \eqref{eq:17} and \eqref{eq:18}.
\item We also have a natural morphism $b:\mathcal{M}\to  h^{1}/h^{0}(E_{\bullet}^{\vee})$ given by the zero section.
\end{enumerate}

From these two morphisms, we can perform the homotopical fiber product
\begin{align}\label{eq:1}
  \xymatrix{\mathcal{M}\times_{h^{1}/h^{0}(E^{\vee}_{\bullet})}^{h}C_{\mathcal{M}} \ar[r]\ar[d]^{r}& C_{\mathcal{M}}\ar[d] \\ \mathcal{M}
    \ar[r] &h^{1}/h^{0}(E_{\bullet}^{\vee})}
\end{align}
As the standard fiber product is $\mathcal{M}$, we have that
$\mathcal{M}\times_{h^{1}/h^{0}(E^{\vee}_{\bullet})}^{h}C_{\mathcal{M}}$ is a derived enhancement of $\mathcal{M}$ with
$\widetilde{j}:\mathcal{M}\to
\mathcal{M}\times_{h^{1}/h^{0}(E^{\vee}_{\bullet})}^{h}C_{\mathcal{M}}$ the canonical closed
embedding. Notice that in the case $\mathcal{M}=\overline{\mathcal{M}}_{g,n}(X,\beta)$, we get a
 derived enhancement which is different from $\mathbb{R}\overline{\mathcal{M}}_{g,n}(X,\beta)$ (see
 Remark \ref{rk:different,retract}). We
 will compare these two structures in \S~\ref{sec:comp-theor-two}.
Hence we can apply the Lemma \ref{lem:K,iso} and we denote
\begin{align}\label{eq:vir,POT}
  [\mathcal{O}_{\mathcal{M}}^{\vir,\POT}]:=\widetilde{j}_{*}^{-1}[\mathcal{O}_{\mathcal{M}\times_{h^{1}/h^{0}(E^{\vee}_{\bullet})}^{h}C_{\mathcal{M}}}]
  \in G_0(\mathcal{M})
\end{align}
where POT means Perfect Obstruction Theory.
The definition of Lee for the virtual sheaf turns to be exactly this one. Indeed, Lee consider the
following (not homotopical) la
cartesian  diagram
\begin{align}\label{eq:diag,POT}
  \xymatrix{\mathcal{M}\times_{E_{-1}^{\vee}} C_{1} \ar[r] \ar[d]^{r}&C_{1}\ar[r] \ar[d]& C_{\mathcal{M}}\ar[d] \\ \mathcal{M}\ar[r]&
    E_{-1}^{\vee}\ar[r]&h^{1}/h^{0}(E_{\bullet}^{\vee})}
\end{align}

In \cite[p.8]{MR2040281}, Lee takes as a definition for the virtual sheaf 
\begin{align*}
\mathcal{O}_{\mathcal{M}}^{\vir}:= \sum_{i}(-1)^{i}
\mathcal{T}or_{i}^{h^{1}/h^{0}}(\mathcal{O}_{\mathcal{M}},\mathcal{O}_{C_{1}})= \mathcal{O}_{\mathcal{M}}\otimes^{\mathbb{L}}_{h^{1}/h^{0}}\mathcal{O}_{C_{1}}=\mathcal{O}_{\mathcal{M}}^{\vir,\POT}
\end{align*}
where the last equality follows from Lemma \ref{lem:K,iso}.

\subsection{Comparison theorem of the two approachs}
\label{sec:comp-theor-two}

Let $\mathcal{M}:=\overline{\mathcal{M}}_{0,n}(X,\beta)$. In this section, we want to compare
$\mathcal{O}_{\mathcal{M}}^{\vir,\DAG}$ with $\mathcal{O}_{\mathcal{M}}^{\vir,\POT}$. The first question is : what is
the perfect obstruction theory we are choosing ?

This is given by the following result.
\begin{prop}[\cite{2011-Schur-Toen-Vezzosi}]
  Let $\mathbb{R}\mathcal{M}$ be a derived Deligne-Mumford stack. Denote by $\mathcal{M}$ its truncation and
  its truncation morphism by $j:\mathcal{M}\hookrightarrow \mathbb{R}\mathcal{M}$. Then
$j^{*}\mathbb{L}_{\mathbb{R}\mathcal{M}}\to \mathbb{L}_{\mathcal{M}}$ is a perfect obstruction theory.
\end{prop}

Now the original question makes perfectly sense and we have the following result that says that they
are the same sheaves.

\begin{thm}[See Proposition 4.3.2 in \cite{2015arXiv150502964M}]\label{thm,O,pot=Dag}
In $G_0(\mathcal{M})$, we have
\begin{displaymath}
  [\mathcal{O}_{\mathcal{M}}^{\vir,\DAG}] =[\mathcal{O}_{\mathcal{M}}^{\vir,\POT}]
\end{displaymath}
\end{thm}

\begin{remark}\label{rk:different,retract}
  Notice that the two enhancements $\mathbb{R}\mathcal{M}$ or
  $\mathcal{M}\times_{h^{1}/h^{0}(E^{\vee}_{\bullet})}^{h}C_{\mathcal{M}}$ are not the same. Indeed,
  the second one has a retract
  $r:\mathcal{M}\times_{h^{1}/h^{0}(E^{\vee}_{\bullet})}^{h}C_{\mathcal{M}}\to \mathcal{M}$ given in
  the diagram \eqref{eq:1} that is $r \circ \widetilde{j} =\Id_{\mathcal{M}}$ where $\widetilde{j}$
  is the closed immersion from $\mathcal{M}$ to
  $\mathcal{M}\times_{h^{1}/h^{0}(E^{\vee}_{\bullet})}^{h}C_{\mathcal{M}}$.  From this we get the
  following exact triangle of cotangent complexes
\begin{align}
&  \mathbb{L}_{\widetilde{j}}[-1]\to  \widetilde{j}^{*}\mathbb{L}_{\mathcal{M}\times_{h^{1}/h^{0}(E^{\vee}_{\bullet})}^{h}C_{\mathcal{M}}}
   \to \mathbb{L}_{\mathcal{M}}  \label{eq:6}\\
& r^{*}\mathbb{L}_{\mathcal{M}}\to
\mathbb{L}_{\mathcal{M}\times_{h^{1}/h^{0}(E^{\vee}_{\bullet})}^{h}C_{\mathcal{M}}}
\to\mathbb{L}_{r} \label{eq:4}
\end{align}
Applying $\widetilde{j}^{*}$ to the second line, we get
\begin{displaymath}
   \mathbb{L}_{\mathcal{M}}\to
\widetilde{j}^{*}\mathbb{L}_{\mathcal{M}\times_{h^{1}/h^{0}(E^{\vee}_{\bullet})}^{h}C_{\mathcal{M}}}
\to \widetilde{j}^{*}\mathbb{L}_{r} \label{eq:4}
\end{displaymath}
This means that \eqref{eq:6} has a splitting that is
\begin{align}\label{eq:spli}
  \widetilde{j}^{*}\mathbb{L}_{\mathcal{M}\times_{h^{1}/h^{0}(E^{\vee}_{\bullet})}^{h}C_{\mathcal{M}}}=\mathbb{L}_{\widetilde{j}}[-1]\oplus \mathbb{L}_{\mathcal{M}}
\end{align}
Comparing to the cotangent complex of $\mathbb{R}\mathcal{M}$ that has no reason to split, we get a
priori two different derived enhancement of $\mathcal{M}$.  
\end{remark}



Notice that in the work of  Fantechi-G\"ottsche \cite[Lemma
3.5]{MR2578301} (see also  Roy Joshua \cite{Roy-joshua}), they prove
that for a scheme $X$ with a perfect obstruction theory $E^{\bullet}:=[E^{-1}\to E^{0}]$, we have 
\begin{align}
  \label{eq:20}
\tau_{X}(\mathcal{O}_{X}^{\vir,\POT})=\Td (TX^{\vir})\cap [X^{\vir,\POT}]
\end{align}
where $TX^{\vir}\in G_{0}(X)$ is the class of $[E_{0}]-[E_{1}]$ where $[E_{0}\to E_{1}]$
is the dual complex of $E^{\bullet}$ and $\tau_{X}:G_{0}(X) \to A_{*}(X)_{\mathbb{Q}}$. 

Notice that the Formula \eqref{eq:20} with Theorem \ref{thm,O,pot=Dag} implies that
\begin{displaymath}
[\overline{\mathcal{M}}_{g,n}(X,\beta)]^{\vir,\POT}=\tau(\mathcal{O}_{\mathbb{R}\overline{\mathcal{M}}_{g,n}(X,\beta)})\Td(T_{\mathbb{R}\overline{\mathcal{M}}_{g,n}(X,\beta)})^{-1}
\end{displaymath}

\appendix


\section{Proof of theorem \ref{thm:colim}}
\label{sec:colimit-proof}

\begin{thm}
The map 
$$
f: \colim^{\mathrm{DM}}\, X_{\bullet, \beta} \to Z_\beta
$$
\noindent  of \cite[(4.2.9)]{2015arXiv150502964M} is an equivalence of derived Deligne-Mumford stacks.
\begin{proof}
It follows from the discussion in the proof of  \cite[Prop. 4.2.1]{2015arXiv150502964M} that 

\BarrBeckTriangle{\Perf(Z_\beta); \Perf(\colim^{\mathrm{DM}}\, X_{\bullet, \beta}); \lim_{\Delta} \Perf(X_{\bullet, \beta});f^*;g;h}

\noindent commutes with the morphism $h$ being an equivalence after h-descent for perfect complexes
\cite[4.12]{1402.3204} and the morphism $g$ being fully faithful after the result of gluing along
closed immersions \cite[16.2.0.1]{Lurie-SAG}. This immediately implies that the map $f^*$ is an
equivalence of categories because we have $g\circ f^{*}=h$ and $g$ is conservative as it is fully faithful.

As both source and target of $f$ are perfect stacks (the first being a colimit of perfect stacks
along closed immersions and second being pullback of perfect stacks), $f^*$ induces an equivalence

$$
\xymatrix{\Qcoh(Z_\beta)\ar[r]^-{f^*}&\Qcoh(\colim^{\mathrm{DM}}\, X_{\bullet, \beta})}$$

We conclude that $f$ is an equivalence using Tannakian duality \cite[9.2.0.2 ]{Lurie-SAG}.\\

\end{proof}
\end{thm}


\section{Proof of Theorem \ref{thm:orientation}.(1)}
\label{sec:proof-theorem}

Let $X$ be a derived stack. We will use the linear derived stacks $\mathbb{V}(\mathbb{E})$ (See
\cite[p.200]{MR3285853} ) where $\mathbb{E}$ is a complex of quasi-coherent sheaf on $X$.  We have a
morphism $\mathbb{V}(\mathbb{E})\to X$ and a zero section $s:X\to \mathbb{V}(\mathbb{E})$.
One should understand that $\mathbb{V}(\mathbb{E})$ as a vector bundle where the fibers are
$\mathbb{E}$. 

It is a derived generalisation of $\mathbf{\Spec} \Sym \mathcal{E}$ for a coherent sheaf
$\mathcal{E}$. If $\mathbb{E}$ is a two terms complex with cohomology in degree $0$ and $1$, then we
have that $t_{0}(\mathbb{V}(\mathbb{E}^{\vee}[-1]))=[h^{1}/h^{0}(\mathbb{E})]$ (See \S 2 in
\cite{MR1437495} for the definition of the quotient stacks).

  Let recall some notation of \S \ref{sec:virtual-object-from} and \S \ref{sec:comp-theor-two}.  Let
  $g,n \in\mathbb{N}$ and $\beta\in H_{2}(X,\mathbb{Z})$.  Denote by $j$ the closed immersion
  $\overline{\mathcal{M}}_{g,n}(X,\beta)\to \mathbb{R}\overline{\mathcal{M}}_{g,n}(X,\beta)$. To
  simplify the notation, put $\mathcal{M}=\overline{\mathcal{M}}_{g,n}(X,\beta)$ and
  $\mathbb{R}\mathcal{M}=\mathbb{R}\overline{\mathcal{M}}_{g,n}(X,\beta)$.

From the exact triangle
\[
j^{*}\mathbb{L}_{\mathbb{R}\mathcal{M}} \to \mathbb{L}_{\mathcal{M}} \to \mathbb{L}_{j}
\]
We deduce that following cartesian diagram 
  \begin{align}
     \xymatrix{
\mathbb{V}(\mathbb{L}_{j}[-1]) \ar[r] \ar[d]& \mathbb{V}(\mathbb{L}_{\mathcal{M}}[-1]) \ar[d] \\ \mathcal{M}
       \ar[r]& \mathbb{V}(j^{*}\mathbb{L}_{\mathbb{R}\mathcal{M}}[-1])
}
   \end{align}
Recall that $j^{*}\mathbb{L}_{\mathbb{R}\overline{\mathcal{M}}_{g,n}(X,\beta)}$ is a two terms
complex in degree $-1$ and $0$ but in general it is not the case for $\mathbb{L}_{j}$ and $\mathbb{L}_{\overline{\mathcal{M}}_{g,n}(X,\beta)}$.
Comparing with Behrend-Fantechi, we have $   t_{0}(\mathbb{V}(\mathbb{L}_{\mathcal{M}}[-1]))$ is the
intrinsic normal sheaf $N_{\mathcal{M}}$ (See \S \ref{sec:virtual-object-from}) and we have the
following cartesian diagram 

  \begin{align}\label{diag:cone}
     \xymatrix{\mathcal{M}\times^{h}_{\mathbb{V}(j^{*}\mathbb{L}_{\mathbb{R}\mathcal{M}}[-1])}
    C_{\mathcal{M}} \ar[r] \ar[d]& C_{\mathcal{M}} \ar[d]\\
\mathbb{V}(\mathbb{L}_{j}[-1]) \ar[r] \ar[d]& \mathbb{V}(\mathbb{L}_{\mathcal{M}}[-1]) \ar[d] \\ \mathcal{M}
       \ar[r]& \mathbb{V}(j^{*}\mathbb{L}_{\mathbb{R}\mathcal{M}}[-1])
}
   \end{align}

\begin{prop}\label{prop:def,normal,sheaf}
  Let $g,n \in\mathbb{N}$ and $\beta\in H_{2}(X,\mathbb{Z})$. 
Denote by $j$ the closed immersion $\overline{\mathcal{M}}_{g,n}(X,\beta)\to
\mathbb{R}\overline{\mathcal{M}}_{g,n}(X,\beta)$ and by
$s:\overline{\mathcal{M}}_{g,n}(X,\beta)\to \mathbb{V}({\mathbb{L}_{j}[-1]})$ be the zero section.
We have the following equality in $G_{0}(\overline{\mathcal{M}}_{g,n}(X,\beta))$
  \begin{align*}
    \mathcal{O}_{\overline{\mathcal{M}}_{g,n}(X,\beta)}^{\vir,\DAG}:=j_{*}^{-1}    \mathcal{O}_{\mathbb{R}\overline{\mathcal{M}}_{g,n}(X,\beta)} = s_{*}^{-1}(\mathcal{O}_{\mathbb{V}(\mathbb{L}_{j})[-1]})
  \end{align*}
\end{prop}

\begin{proof}
From Gaitsgory (see Proposition 2.3.6 p 18 Chapter IV.5 \cite{gaitsgory}), we can construct an derived stack $\mathcal{Y}_{scaled}$ such that the following
diagram has two homotopical fiber products
\begin{align*}
  \xymatrix{\mathbb{R}\mathcal{M} \ar[r]^-{h}& \mathcal{Y}_{scaled}& \ar[l]_-{v} \mathbb{V}(\mathbb{L}_{j}[-1])\\
\mathcal{M}\times \{0\}\ar[r]^-{i_{0}} \ar[u]^{j} & \mathcal{M}\times \mathbb{A}^{1} \ar[u]_{\sigma}&\ar[l]_-{i_{1}}
\mathcal{M}\times \{1\} \ar[u]^{s}}
\end{align*}

We have
\begin{align*}
  (s_{*})^{-1}\mathcal{O}_{\mathbb{V}(\mathbb{L}_{j}[-1])}&
= (s_{*})^{-1}v^{*}\mathcal{O}_{\mathcal{Y}_{scaled}}\\
&=i_{1}^{*}(\sigma_{*})^{-1}\mathcal{O}_{\mathcal{Y}_{scaled}}\\
&=i_{0}^{*}(\sigma_{*})^{-1}\mathcal{O}_{\mathcal{Y}_{scaled}}\\
\end{align*}
The last equality follows from the $\mathbb{A}^{1}$- invariance of the $G$-theory. That is, we have that $G_{0}(\mathcal{M}\times
\mathbb{A}^{1}) \to G_{0}(\mathcal{M})$ and $i_{0}^{*}=(\pi^{*})^{-1}=i_{1}^{*}$ where $\pi$ is the projection.
Applying the same computation as above with the other homotopical fiber product, we get  Formula.
\end{proof}

\begin{remark}
  This statement is a first step in proving Theorem  \ref{thm,O,pot=Dag}. The last step is to prove
  that the inclusion $C_{\mathcal{M}}\to N_{\mathcal{M}}$ induces an equality of the structure sheaf
  in $G_{0}$-theory.
\end{remark}

\begin{cor}\label{prop:degree,0}
For stable maps of degree $0$, we have that 
\[
\mathcal{O}_{\overline{\mathcal{M}}_{g,n}(X,0)}^{\vir,\DAG}=\sum_{i}(-1)^{i}\wedge^{i}(TX\boxtimes R^{1}\pi_{*}\mathcal{O}_{\mathcal{C}})
\]
\end{cor}

\begin{remark}
  Notice that in the case of $\beta=0$, we have that $\overline{\mathcal{M}}_{g,n}(X,\beta=0)
  =\overline{\mathcal{M}}_{g,n}\times X$ which is smooth. Nevertheless, it has a derived
  enhancement, given by the $\mathbb{R}\Map$ which has a retract given by the projection and the
  evaluation. For $\beta\ne 0$, this retract does not exist.

\end{remark}

\begin{proof}

For $\beta=0$, the smoothness of $\mathcal{M}$ implies that the intrinsic normal cone is the
intrinsic normal sheaf that is we have the 
$C_{\mathcal{M}}=\mathbb{V}(\mathbb{L}_{\mathcal{M}}[-1])$ in the diagram \eqref{diag:cone}.
The second thing which is different is that $j:\mathcal{M}\to \mathbb{R}\mathcal{M}$ has a
retract. This implies that $\mathbb{L}_{j}[-1]\simeq\mathbb{L}_{\mathcal{M}}[-1]\oplus j^{*}\mathbb{L}_{\mathbb{R}\mathcal{M}}$.
Hence the Proposition \ref{prop:def,normal,sheaf}, implies that we need to compput
$s_{*}^{-1}\mathcal{O}_{\mathbb{V}(\mathbb{L}_{j}[-1])}$ which is by standard computation
$\sum_{i}(-1)^{i}\wedge^{i}(TX \boxtimes R^{1}\pi_{*}\mathcal{O}_{\mathcal{C}})$ where
$\mathcal{C}$ is the universal curve of $\overline{\mathcal{M}}_{g,n}$.


\end{proof}

From the proof, we see that the RHS of the formula is the structure sheaf of
$\mathbb{V}(\mathbb{L}_{j}[-1])$. In fact, we think that
$\mathbb{R}\overline{\mathcal{M}}_{g,n}(X,0)$ is isomorphic to
$\mathbb{V}(\mathbb{L}_{j}[-1])$. This should follow from a general argument that we will detail in
the next section for the affine case.




\section{Alternative proof of Corollary \ref{prop:degree,0} in the affine case.}

\begin{prop}\label{prop:qsmooth-dagvectbundle}
Let $F=\Spec A$ be an affine quasi-smooth algebraic derived stack. Let $F_{0}=\Spec \pi_{0}(A)$ its
truncation and denote $j:F_{0}\to F$ its closed immersion. Assume that $F_{0}$ is smooth and that $F$ admit a retract
$r:F\to F_{0}$. Then $F=\mathbb{V}(\mathbb{L}_{j}[-1])$. 
\end{prop}
 
This proposition is a way of proving Corollary \ref{prop:degree,0} in the affine case without using
the deformation argument of Gaitsgory. We believe that we can drop the affine assumption in the
previous proposition.

Notice that we can drop the existence of the retract in the hypothesis because when $F{0}$ is
smooth, there always exists a retract (see the Remark \ref{rmk:retract}).



\medskip
\begin{lem}\label{lem:appen}
  With the previous hypothesis, we have
  \begin{align*} 
 \pi_{0}(\mathbb{L}_{j})=\pi_{1}(\mathbb{L}_{j})=0\\
    \pi_{2}(\mathbb{L}_{j})=\pi_{1}(j^{*}\mathbb{L}_{F})=\pi_{2}(\mathbb{L}_{\pi_{0}(A)/\tau_{\leq
      1}A}) = \pi_{1}(A)\\
\mathbb{L}_{j}[-1] \simeq \pi_{1}(A)[1]
  \end{align*}
\end{lem}

\begin{proof}
We have the triangle
\begin{displaymath}
  j^{*}\mathbb{L}_{F}\to\mathbb{L}_{F_{0}} \to \mathbb{L}_{j}.
\end{displaymath}

Applying the hypothesis, we get
\begin{enumerate}
\item As $F$ is quasi-smooth, we have that $\pi_{2}(j^{*}\mathbb{L}_{F})=0$.
\item As $F_{0}$ is smooth, we have that $\pi_{2}(\mathbb{L}_{F_{0}})=\pi_{1}(\mathbb{L}_{F_{0}})=0$.
\item As $j^{*}\mathbb{L}_{F}\to \mathbb{L}_{F_{0}}$ is a perfect obstruction theory, we deduce
  $\pi_{0} (j^{*}\mathbb{L}_{F}) \simeq \pi_{0}(\mathbb{L}_{F_{0}})$ and $\pi_{1}
  (j^{*}\mathbb{L}_{F}) \to \pi_{1} (\mathbb{L}_{F_{0}})$ is onto. 
\end{enumerate}
        
Applying the three properties above to the associated long exact sequence, we get
\begin{enumerate}
\item As $F$ is quasi-smooth, we have that $\pi_{2}(j^{*}\mathbb{L}_{F})=0$.
\item As $F_{0}$ is smooth, we have that $\pi_{2}(\mathbb{L}_{F_{0}})=\pi_{1}(\mathbb{L}_{F_{0}})=0$.
\item As $j^{*}\mathbb{L}_{F}\to \mathbb{L}_{F_{0}}$ is a perfect obstruction theory, we deduce
  $\pi_{0} (j^{*}\mathbb{L}_{F}) \simeq \pi_{0}(\mathbb{L}_{F_{0}})$ and $\pi_{1}
  (j^{*}\mathbb{L}_{F}) \to \pi_{1} (\mathbb{L}_{F_{0}})$ is onto. 
\end{enumerate}

\begin{tikzpicture}
\matrix[matrix of nodes,ampersand replacement=\&, column sep=0.5cm, row sep=0.5cm](m)
{
        
 $0$ \& $0 $ \& $\pi_{2} (\mathbb{L}_{j})$ \\
 $\pi_{1} (j^{*}\mathbb{L}_{F})$ \& $0$ \& $0$ \\
 $ \pi_{0} (j^{*}\mathbb{L}_{F})$ \& $\pi_{0}(\mathbb{L}_{F_{0}})$ \& $0$ \\
};
\draw[->] (m-1-1) edge (m-1-2)
          (m-1-2) edge (m-1-3)
          (m-1-3) edge[out=0, in=180] (m-2-1)
          (m-2-1) edge (m-2-2)
          (m-2-2) edge (m-2-3)
          (m-2-3) edge[out=0, in=180] (m-3-1)
          (m-3-1) edge    (m-3-2)
          (m-3-2) edge (m-3-3);
\end{tikzpicture}

We conclude that
\begin{enumerate}
\item  $ \pi_{2}(\mathbb{L}_{j})=\pi_{1}(j^{*}\mathbb{L}_{F})$
\item $\mathbb{L}_{j}$ is $2$-connective.
\end{enumerate}
To prove the second equality of the lemma, we use the Postnikov tower that is we consider the
closed immersion $j_{1}:F_{0}\to F_{1} $ and $j_{2}:F_{1}\to F$ where $F_{1}$ is $\Spec \tau_{\leq 1} A$. We deduce the exact triangle
\begin{align*}
  j_{1}^{*}\mathbb{L}_{j_{2}} \to \mathbb{L}_{j} \to \mathbb{L}_{j_{1}}
\end{align*}

As  we have $j$ and $j_{1}$ are $1$-connected and $j_{2}$ is $2$-connected, we deduce from
connective estimates  that $\mathbb{L}_{j}$ and $\mathbb{L}_{j_{1}}$ are
$2$-connective and $\mathbb{L}_{j_{2}}$ is $3$-connective (See Corollary 5.5 in \cite{2013arXiv1310.3573P}). We deduce from the long exact sequence
that $\pi_{2}(\mathbb{L}_{j})=\pi_{2}(\mathbb{L}_{j_{1}})$. How we apply Lemma 2.2.2.8 in
\cite{Toen-Vezzosi-2008-HAGII} that implies that $\pi_{2}(\mathbb{L}_{j_{1}})=\pi_{1}(A)$.

As we have that $\pi_{k}(\mathbb{L}_{j})=0$ for all $k\neq 2$ and
$\pi_{2}(\mathbb{L}_{j})=\pi_{1}(A)$, we deduce that $\mathbb{L}_{j}[-1]\simeq \pi_{1}(A)[1]$.
\end{proof}

\begin{proof}[Proof of Proposition \ref{prop:qsmooth-dagvectbundle}]
To prove the proposition,  we will show that
\begin{align}\label{eq:25}
  B:=\Sym_{\pi_{0}(A)} (\pi_{1}(A)[1]) \simeq A
\end{align}
First, we will construct a morphism $f:B\to A$. Notice that $\pi_{1}(A)$ is a free $\pi_{0}(A)$
module by the last statement of Lemma \ref{lem:appen}. Then we get an inclusion $\pi_{1}(A)[1]\to A$
of $\pi_{0}(A)$-modules which induces $f:B\to A$. Moreover $f$ is an equivalence on $\pi_{0}$ and
$\pi_{1}$ that is $\tau_{\leq 1}B\simeq \tau_{\leq 1}A$.

Then we construct an inverse from $A\to B$ using the Postnikov tower. We have $\varphi:A\to\tau_{\leq
  1}A\simeq \tau_{\leq 1}B$. As $B$ is the colimit of its
Postnikov tower, we will proceed by induction on the Postnikov tower. First, we want to lift the
morphism $\varphi:A \to \tau_{\leq 1}B$ to $A\to \tau_{\leq 2}B$. We use the following cartesian diagram
(See Remark 4.3 in \cite{2013arXiv1310.3573P})
\begin{align}\label{eq:26}
  \xymatrix{\tau_{\leq 2}B \ar[r] \ar[d]& \tau_{\leq 1}B\ar[d]^{d} \\ \tau_{\leq 1} B\ar[r]^-{\Id,0} &\tau_{\leq 1}B \oplus \pi_{2}(B)[3]}
\end{align}
Hence, we need to construct a commutative diagram 
\begin{align}\label{eq:27}
  \xymatrix{A \ar[r]^{\varphi} \ar[d]^{\varphi}& \tau_{\leq 1}B\ar[d]^{d} \\ \tau_{\leq 1} B\ar[r]^-{\Id,0} &\tau_{\leq 1}B \oplus \pi_{2}(B)[3]}
\end{align}
As $\mathbb{L}_{A}$ has a tor-amplitude in $[-1,0]$, we have that
\begin{align*}
  \pi_{0}(\Map(\mathbb{L}_{A},\pi_{2}(B)[3]))=0\\
\pi_{1}(\Map(\mathbb{L}_{A},\pi_{2}(B)[3]))=0
\end{align*}
Hence we deduce a morphism from $\psi:A\to A\oplus_{d\circ \varphi}\pi_{2}(B)[3]$.  Hence we get the
morphism from $A\to B_{\tau_{\leq 2}}$.
\begin{align*}
  \xymatrix{
A \ar@/^/[rrrrd] \ar@/_/[ddddr] \ar@{.>}[rd]^-{\psi}&&&&&\\
& A\oplus_{d\circ \varphi}\pi_{2}(B)[3] \ar[ddd]\ar[rrr] \ar[rd]&&& A \ar[ddd]^{d\circ \varphi} \ar[ld]^-{\varphi}& \\
&    &B_{\tau_{\leq 2}}\ar[r] \ar[d]& B_{\tau{\leq 1}} \ar[d]^-{d} &&\\
&&B_{\tau_{\leq 1}} \ar[r]^-{0}&B_{\tau_{\leq 1}}\oplus \pi_{2}(B)[3]&&\\
&A\ar[rrr]^-{0}\ar[ru]^-{\varphi}&&&A\oplus \pi_{2}(B)[3] \ar[lu]^-{\varphi,\Id}&\\
}
\end{align*}

Hence by induction, we get a morphism from $g:A\to B$.
The composition $g\circ f: B\to A\to B$ is  the identity on $\pi_{1}(B)$ and by the universal
property of $\Sym$, we deduce that $g\circ f=\Id_{B}$. This implies that $\pi_{i}(B)=\wedge^{i}\pi_{1}(A)\to \pi_{i}(A)$
is injective. To finish the proof, we will prove that these morphisms are surjective. 

For this purpose we use another characterization of afffine quasi-smooth derived scheme. Let us fix
generators of $\pi_0(A)$. This choice is determined a surjective map of commutative $k$-algebras $k[x_1,.., x_n]\to
\pi_0(A)$. As the polynomial ring is smooth, we proceed by induction on the Postnikov tower of $A$
to construct a morphism from $k[x_{1}, \ldots ,x_{n}]\to \tau_{\leq n}A$. We use the same idea as
above for constructing the morphism $A\to B$.
We get a map of cdga's $k[x_1,.., x_n]\to A$ which remains a closed immersion.  Moreover, one can now choose
generators for the kernel $I$ of $k[x_{1}, \ldots ,x_{n}]\to \pi_0(A)$, say, $f_1,.., f_m$ whose image in
$I/I^2$ form a basis. The fact that $k[y_1,.., y_m]$ is smooth allows us to extend the zero composition
map
$$
k[y_1,..., y_m]\to k[x_1,.., x_m]\to \pi_0(A)
$$
to map
$$
k[y_1,..., y_m]\to k[x_1,.., x_m]\to A
$$
together with a null-homotopy. This puts $A$ in a commutative square of cdga's
\begin{displaymath}
  \xymatrix{k[y_{1}, \ldots ,y_{m}]\ar[r] \ar[d] & k[x_{1}, \ldots ,x_{n}] \ar[d] \\ k \ar[r]& A}
\end{displaymath}
which we is a pushout square. Indeed, it suffices to show that the canonical map
$$
k\otimes^{\mathbb{L}}_{k[y_1,.., y_m]}k[x_1,..., x_n]\to A
$$
induces an isomorphism between the cotangent complexes. But as $\Spec(A)$ is quasi-smooth, its
cotangent complex is perfect in tor-amplitudes $-1, 0$, meaning that it can  be written as 
$$
{A}^m\to {A}^n
$$
and this identifies with the standard description of the cotangent complex of the derived tensor product $k\otimes^{\mathbb{L}}_{k[y_1,.., y_m]}k[x_1,..., x_n]$.
This implies that surjectivity of the morphisms $\pi_{i}(B)\to \pi_{i}(A)$. 
\end{proof}

\begin{remark}\label{rmk:retract}
  As $F=\Spec A$ is a derived scheme (not necessarily quasi-smooth) and its truncation is $F_0$ is smooth,
  we have that $F_{0}\to F$ admits a retract.  We proceed by induction on the
  Postnikov tower of $A$ to construct a lift
$$
\xymatrix{
& A\ar[d]\\
\pi_0(A)\ar[r]^{\Id}\ar@{->}[ur]&\pi_0(A) 
}
$$
We use the same kind of diagrams as \eqref{eq:26} and \eqref{eq:27}
Indeed, as $\mathbb{L}_{F_0}$ is concentrated in degree 0, all the groups 
$$
\pi_0(\Map(\mathbb{L}_{F_0}, \pi_n(A)[n+1]))=\pi_1(\Map(\mathbb{L}_{F_0}, \pi_n(A)[n+1]))=0
$$
vanish for $n\geq 1$ saying that the liftings exist at each level of the Postnikov tower the space of choices of such liftings is connected. 
\end{remark}

\bibliographystyle{alpha}	
\bibliography{biblio}

\end{document}